\documentclass{article}

\usepackage{amsfonts,amssymb, latexsym, amsmath, amsthm,mathrsfs, verbatim,geometry}

\usepackage{slashed}
\usepackage{amsfonts}
\usepackage{url,color}
\usepackage{enumerate}
\usepackage{esint}
\usepackage{pifont}
\usepackage{upgreek}
\usepackage{times}
\usepackage{calligra}
\usepackage{graphicx}
\usepackage{caption}
\usepackage{amscd}
\usepackage{amsthm}
\usepackage{amsmath,bm}
\usepackage{amssymb}
\usepackage{comment}
\usepackage{tikz}
\usepackage[utf8]{inputenc}
\usepackage{gensymb}
\usepackage{pifont}
\usepackage{autobreak}
\usepackage{color}
\usepackage{verbatim}
\usepackage{systeme}
\usepackage[font=scriptsize]{caption}
\usepackage[ngerman, english]{babel}
\usepackage[title,toc,page]{appendix}
\usepackage{subfig}
\usepackage{tikz}
\usetikzlibrary{calc}
\usetikzlibrary{calc}
\usetikzlibrary{arrows.meta}
\usetikzlibrary{hobby}
\usepackage{tkz-euclide}
\usepackage{mathtools}

\newcommand{\Jp}{\mathcal J}
\newcommand{\gaSix}{\gamma_\star}

\newcommand{\Barrier}{\mathfrak B}

\newcommand{\gau}{\gamma_u}

\newcommand{\ga}{\gamma}
\newcommand{\la}{\lambda}

\newcommand{\al}{\alpha}
\newcommand{\Pg}{\mathring{P}}
\newcommand{\Cg}{\mathring{C}}
\newcommand{\Vg}{\mathring{V}}

\def\bcr{\begin{color}{red}}
\def\bcb{\begin{color}{blue}}
\def\bcc{\begin{color}{violet}}
\def\ec{\end{color}}

\numberwithin{equation}{section}

\usepackage{listings}

\allowdisplaybreaks

\usepackage{hyperref}
\hypersetup{
    colorlinks,
    citecolor=red,
    filecolor=green,
    linkcolor=blue,
    urlcolor=black
}
\definecolor{backcolour}{rgb}{0.95,0.95,0.92}
\lstdefinestyle{mystyle}{
    backgroundcolor=\color{backcolour},   
}

\usepackage{tabularx}
\usepackage{multirow}
\usepackage{blindtext} 
\usepackage{charter} 
\usepackage{microtype} 
\usepackage[english, ngerman]{babel} 
\usepackage{amsthm, amsmath, amssymb} 
\usepackage{float} 
\usepackage{graphicx, multicol} 
\usepackage{xcolor} 
\usepackage{marvosym, wasysym} 
\usepackage{rotating} 
\usepackage{censor} 
\usepackage{pseudocode} 
\usepackage{booktabs} 
\usepackage{csquotes} 
\newtheorem{theorem}{Theorem}[section]

\newtheorem{lemma}[theorem]{Lemma}
\newtheorem{remark}[theorem]{Remark}

\newtheorem{proposition}[theorem]{Proposition}

\title{Converging/diverging self-similar shock waves: from collapse to reflection
}
\author{Juhi Jang\thanks{Department of Mathematics, University of Southern California, Los Angeles, CA 90089, USA, and Korea Institute for Advanced Study, Seoul, Korea.  Email: juhijang@usc.edu.}, \  Jiaqi Liu\thanks{Department of Mathematics, University of Southern California, Los Angeles, CA 90089, USA.  Email: jiaqil@usc.edu.} \ and Matthew Schrecker\thanks{Department of Mathematics, University of Bath, Claverton Down, Bath, UK. Email: m.schrecker@ucl.ac.uk.}}

\date{}
\begin{document}

\maketitle

\abstract{We solve the continuation problem for the non-isentropic Euler equations following the collapse of an imploding shock wave. More precisely, we prove that the self-similar G\"uderley imploding shock solutions for a perfect gas with adiabatic exponent $\gamma\in(1,3]$ admit a self-similar extension consisting of two regions of smooth flow separated by an outgoing 
spherically symmetric shock wave of finite strength. In addition, for $\gamma\in(1,\frac53]$, we show that there is a unique choice of shock wave that gives rise to a globally defined self-similar flow with physical state at the spatial origin.
}

\tableofcontents

\section{Introduction}

The converging/diverging shock wave problem in gas dynamics is a classical one, dating back to the works of G\"uderley, Landau and Stanyukovich \cite{Guderley42,Stanyukovich60}. A spherical shock wave moving radially inwards from infinity converges to the spatial origin. As it converges, the shock strengthens and, at the point of collapse, the strength of the shock becomes infinite. A new, outgoing shock wave of finite strength then forms at the centre of symmetry and propagates outwards to infinity. 
The theory of converging/diverging shock waves 
has attracted a great deal of attention due to its many applications from inertial confinement fusion to the destruction of kidney stones, while the combination of a relatively simple structure and relatively severe singularities has made it a very attractive test problem for numerical codes (see, for example, \cite{Bilbao96,Giron23,Lazarus81,Ponchaut06,Ramsey12}). The expansion part of the problem, in particular, plays an important role in determining ignition and moreover, it 
offers an example of a globally defined flow for the Euler equations, 
which has spurred interest from analysts 
\cite{Courant48,Jenssen18,Jenssen23,Landau87}.
In a recent work, \cite{JLS23}, the authors rigorously constructed
the G\"uderley imploding shock solution up to the blow-up time at which the incoming shock wave meets the centre. 
In this paper, we solve the continuation problem past the blow-up time by providing a rigorous construction of the reflected expanding shock wave and the flow on either side for all positive times. 

The Euler system for a compressible perfect gas flow  in radial or cylindrical symmetry is given by the system of PDEs 
\begin{equation}\label{symmetric equations}
	\begin{aligned}
    \rho_t+\frac{1}{r^m}(r^m\rho u)_r&=0,\\
    (\rho u)_t +\frac{1}{r^m}\big(r^m(\rho u^2)\big)_r + p_r &=0,\\
    \Big[\rho\cfrac{u^2}{2}+\frac{p}{\ga-1} \Big]_t + \frac{1}{r^m}\Big[r^m u\big( \rho\cfrac{u^2}{2} +\frac{\ga p}{\ga-1}\big)\Big]_r &=0,    
\end{aligned}
\end{equation}
where $(t,r)\in \mathbb R \times \mathbb R_+$. The functions $\rho = \rho(t,r)\geq 0$, $u=u(r,t)$, $p(t,r)\geq 0$ denote  the density, radial fluid velocity,  and
 pressure, respectively. Here  $\gamma>1$ denotes the adiabatic exponent for the gas  and $m=1,2$ distinguishes flows with cylindrical or spherical symmetry. It is well known that, defining 
the sound speed by $c= \sqrt{{\gamma p/\rho\,}}$,
we may reformulate the Euler equations, when $\rho>0$, as 
\begin{align}
    \rho_t+(\rho u)_r+\cfrac{m\rho u}{r} &=0,\label{rho_t}\\
     u_t +uu_r+\frac{1}{\gamma\rho}(\rho c^2)_r &=0,
     \label{u_t}\\
     c_t+uc_r+\frac{\gamma-1}{2}c(u_r+\frac{mu}{r}) &=0.\label{c_t}
\end{align}
The G\"uderley solutions are constructed as self-similar solutions of the Euler system. Therefore, we make the following self-similar ansatz. For some $\la>1$, we seek the solution to~\eqref{rho_t}--\eqref{c_t} in the form
\begin{align}
    u(t,r) & = -\frac{r}{\lambda t}V(x) = -\frac{r^{1-\lambda}}{\lambda}\frac{V(x)}{x},\label{ansatz for velocity}\\
    c(t,r) & = -\frac{r}{\lambda t}C(x) = -\frac{r^{1-\lambda}}{\lambda}\frac{C(x)}{x},\label{ansatz for sound speed}\\
    \rho(t,r) & = R(x),\label{ansatz for density}
\end{align}
where the self-similar variable 
\begin{equation}\label{def:x}
x=\frac{t}{r^\la}.
\end{equation}
By substituting \eqref{ansatz for velocity}--\eqref{ansatz for density} to the Euler system  \eqref{rho_t}--\eqref{c_t} and performing algebraic cancellations,  we obtain a system of ODEs for the two principal unknowns, $V(x)$, $C(x)$: 
\begin{equation}\label{ODE system}
	\begin{aligned}
		    \cfrac{d V}{d x} &= -\frac{1}{\lambda x}\frac{G(V(x),C(x);\gamma,z)}{D(V(x),C(x))} \ \ \text{ and } \ \ 
     \cfrac{d C}{d x} = -\frac{1}{\lambda x}\frac{F(V(x),C(x);\gamma,z)}{D(V(x),C(x))},
	\end{aligned}
\end{equation}
where 
\begin{align}
   D(V,C) &=(1+V)^2 - C^2, \label{D(V,C)}\\
   G(V,C;\gamma,z) &= C^2g_1(V)-g_2(V)\label{G(V,C)},\\
   F(V,C;\gamma,z) &= C\big(C^2f_1(V)- f_2(V)\big)\label{F(V,C)},
\end{align}
and 
\begin{equation}
\label{g1g2f1f2}
	\begin{aligned}
	g_1(V) =& \,(m+1)V+2mz,\ \ \ &&g_2(V) = V(1+V)(\lambda+V),\\
	 f_1(V)=&\,1+\frac{mz}{(1+V)},\ \ \ && f_2(V)= a_1(1+V)^2-a_2(1+V)+a_3.
	\end{aligned}
\end{equation}
Here, the constants $z$, $a_i$ are determined from $\ga$, $m$, and $\la$ by 
\begin{equation}\label{a_1234&z}
	\begin{aligned}
	z = \cfrac{\lambda -1 }{m\gamma}, \ \ \ a_1 = 1+\cfrac{m(\gamma-1)}{2}, \ \ \ a_2 =\cfrac{m(\gamma-1)+m \gamma z(\gamma-3)}{2}, \ \ \ a_3 = \cfrac{m \gamma z (\gamma-1)}{2}.
	\end{aligned}
\end{equation}
The density $R(x)$ satisfies 
\begin{align}\label{eq: ode of R}
	\frac{1}{R}\frac{dR}{dx} = \frac{m+1}{\la x} \frac{V}{1+V} - \frac{1}{1+V}\frac{dV}{dx}.
\end{align}
The derivation of the ODE system is standard and we have adopted the notation used by Lazarus \cite{Lazarus81}.

In \cite{JLS23}, the authors constructed a piecewise analytic solution of the ODE system~\eqref{ODE system} for $x\in(-\infty,-1)\cup(-1,0)$ with an entropy-admissible self-similar shock at $x\equiv -1$ for a specific choice $\la=\la_{std}(\ga,m)>1$. This self-similar solution gives an entropy solution of the Euler system~\eqref{symmetric equations} for $t<0$, which extends to give terminal data at $t=0$ of the form
\begin{equation}\label{eq:collapsingterminal}
\rho(0,r)=R(0),\quad u(0,r)=-v_1\frac{r^{1-\la}}{\la},\quad c(0,r)=c_1 \frac{r^{1-\la}}{\la},
\end{equation}
for $R(0)>0$ and 
constants $v_1,c_1>0$ depending on $\ga$ and $m$.  This terminal data for the collapsing shock problem gives the initial data for the subsequent expansion problem for the reflected shock. 

We are now in a position to state our main result.

\begin{theorem}
Let $\ga\in(1,3]$, $m=1,2$, and take $\la=\la_{std}(\ga,m)$. Let $(\rho_0,u_0,c_0)(r)$ given by~\eqref{eq:collapsingterminal} be the terminal data of the collapsing G\"uderley shock solution. Then, for $t>0$, there exists a piecewise smooth entropy-admissible weak solution $(\rho,u,c)$ of~\eqref{symmetric equations} with initial data at $t=0$ given by $(\rho_0,u_0,c_0)$. The solution is self-similar (with similarity exponent $\la_{std}$) and consists of a single shock at $r(t)=\big(\frac{ t}{x_H})^{\frac1\la}$ for some $x_H>0$. On either side of the shock, the density $\rho$ is finite. Moreover, for $\ga\in(1,\frac53]$, this is the unique, physically admissible piecewise smooth solution containing a single shock and connecting to the origin via the admissible connection curve $(V_\infty(C),C)$ defined below (see Section \ref{sec:intro_sub}). 
\end{theorem}

To the best of  our knowledge, 
 this theorem gives the first rigorous proof of the existence of the continuation of the G\"uderley collapsing shock solution, and therefore completes the existence theory for the G\"uderley solutions. Among the novelties of the result, we are able to control the continuation of the flow through the first blowup time in the region $r>0$ onto its maximal smooth development, as well as the jump locus of the maximal development. This maximal development is pivotal 
 in order to find the admissible jump discontinuity that allows for the global definition of the solution in the interior region around the spatial origin, behind the reflected shock. As a by-product of our analysis, we are able to show that the condition that prevents the further extension of the maximal development is the occurrence  of a second sonic point in the flow, which cannot be passed smoothly, and that the admissible jump discontinuity must occur before reaching such a sonic point. 
 
 The role of the maximal development is also significant in the continuation problems related to other singular solutions of the Euler equations and variants. For example, in the shock development problem, arising after formation of a shock wave from regular initial data, the outgoing shock is placed within the maximal development of the smooth solution of the Euler equations, see \cite{Buckmaster22,Christodoulou16,Abbrescia22}, and it is again the smooth development of the data that must be matched across the shock at an \textit{a priori} unknown location. In the context of implosion singularities for the Einstein-Euler system and the naked singularity problem, the construction of an outgoing null geodesic relies on an understanding of the maximal development of the smooth solution, \cite{GHJ23}, while the existence of expanding relativistic shock waves in \cite{Alexander22} again requires the maximal development.
 
 In recent years, smooth expanding flows have received significant attention from the mathematical community. Beginning with work in the 1990s \cite{Grassin98,Serre97}, global in time flows of outgoing velocity were constructed which give classical solutions of the Euler equations. 
 In \cite{Sideris17}, so-called \emph{affine} flows whose support grows in time were constructed to the vacuum free boundary Euler problem. 
 These flows were shown to be nonlinearly stable in \cite{Hadzic18} 
 (see also \cite{Hadzic21a,Hadzic21b,Hadzic22}). 
 Unlike in these solutions, the flows constructed in this paper have a shock wave with density jump, but also have a mixed sign for the velocity and may, in some cases, admit stagnation points of the flow as well. We mention also that data with signed (outgoing) velocity in the case of radial flows permits $L^\infty$ radial solutions of the isentropic Euler system rather than the usual finite energy solutions, compare \cite{Huang19,Schrecker20}.

\subsection{Collapsing shock solution and continuation problem}\label{sec:intro_sub}

In order to state a precise version of our main result, we require a more detailed understanding of the phase plane for the trajectories of~\eqref{ODE system}. In this section, we gather the key results and definitions from \cite{JLS23} relating to the collapsing shock solution and properties of the phase plane that we require to prove the existence of the expanding continuation solution. The main part of the analysis is focused on the ODE obtained by combining the two equations of~\eqref{ODE system} to yield
\begin{align}
	\cfrac{dC}{dV} = \cfrac{F(V,C;\gamma,z)}{G(V,C;\gamma,z)}.\label{ODE}
\end{align} 
Following \cite{Lazarus81}, one may easily classify the critical points of this ODE. There are, first, the triple points, at which $F=G=D=0$. In the notation of \cite{JLS23,Lazarus81}, these are $P_2=(-1,0)$ and
\begin{align}
		P_6 &= (V_6,C_6) = (\frac{-1+(\gamma-2)z-w}{2}, 1+V_6),\quad && P_7 =(V_7,C_7)=(V_6,-C_6), \label{P6}\\
		P_8 &=(V_8,C_8) = (\frac{-1+(\gamma-2)z+w}{2}, 1+V_8),\quad && 	P_9 =(V_9,C_9) = (V_8,-C_8), \label{P8}
	\end{align}
where
\begin{align}
	w(z) = \sqrt{1-2(\gamma+2)z+(\gamma-2)^2z^2}\label{w(z)}.
\end{align}
Moreover, from \cite[Remark 2.5, Lemma 2.6]{JLS23}, 
\begin{align}\label{ineq: V'6>0, V'8<0}
V_6'(z)>0,  \ \ V_8'(z)<0, \ \  w'(z)<0, \ \  V_6(z)< V_8(z), \ \  w(z)\in(0,1) \ \text{ for }z\in(0,z_M),
\end{align}
and 
\begin{equation}\label{Id: VCzM}
V_6(z_M)=V_8(z_M) = -\frac{\sqrt 2}{\sqrt 2+\sqrt \ga},\ \ \ C_6(z_M)=C_8(z_M) = \frac{\sqrt \ga}{\sqrt 2+\sqrt \ga} \ \text{ for }z_M=\frac{1}{(\sqrt{2}+\sqrt{\ga})^{2}}<\frac15.
\end{equation}
 In addition, there are double roots, where $F=G=0$ at $P_0=(0,0)$, $P_3=(-\la,0)$, and
\begin{equation}\label{P4}
P_4=(V_4,C_4)=\Big(\frac{-2\lambda}{\gamma+1+m(\gamma-1)},\sqrt{\cfrac{V_4(1+V_4)(\lambda+V_4)}{(m+1)V_4+2mz}}\Big),\quad P_5=(V_5,C_5)=(V_4,-C_4).
\end{equation}
A calculation shows that there exists $\ga_g\in(\frac{5}{2},3)$, such that there exists $z_g$ satisfying $V_4(z_g)=V_6(z_g)=V_8(z_g)$ when $\ga=\ga_g$, and for $\ga\in(1,\ga_g)$
\begin{equation}\label{ga<gag}
V_4> V_6  \ \text{ for }z\in(0,z_g),\qquad V_4< V_6 \ \text{ for }z\in(z_g,z_M],
\end{equation}
and, for $\ga\in(\ga_g,3]$,
\begin{equation}\label{ga>gag}
V_4> V_8 \ \text{ for }z\in(0,z_g),\qquad V_4< V_8  \ \text{ for }z\in(z_g,z_M],
\end{equation}
where
 \begin{align}\label{zg}
	z_g =\begin{cases}
		\cfrac{\sqrt{\gamma^2+(\gamma-1)^2}-\gamma}{\gamma(\gamma-1)} \quad &\text{ when } m=1,\\
		\cfrac{\sqrt{(2\gamma^2-\gamma+1)^2 +2\gamma(\gamma-1)[4\gamma(\gamma-1)+\frac{8}{3}]}-(2\gamma^2-\gamma+1)}{\gamma[4\gamma(\gamma-1)+\frac{8}{3}]}\quad &\text{ when } m=2.
	\end{cases}
\end{align}
A detailed discussion of  $\gamma_g$ and $z_g$ is given in  \cite{Lazarus81}.
Finally, there are two critical points with $C=-\infty$. The first is
\begin{equation}\label{Pinfty}
P_\infty=(\overline V_\infty,-\infty)=\Big(-\frac{2mz}{m+1},-\infty\Big),
\end{equation}
and the second is $\hat{P}=(-1,-\infty)$.

In addition, 
we denote the $\gamma$ and $P_6/P_8$ dependent $z$-intervals containing $z_{std}$ by  
\begin{align}
	\mathring{\mathcal{Z}}(\gamma;P_*) = \begin{cases}
		&(z_g(\gamma),z_M(\gamma)] \quad\text{for }\gamma\in(1,2] \text{ at }P_6,\\
		&(\frac{\sqrt{5}-1}{2(1+\sqrt{5}+\gamma)},z_M(\gamma)] \quad\text{for }\gamma\in(\gaSix,1+\sqrt 2] \text{ at }P_8,\\
		&(\frac{\sqrt{33}-3}{6+2\sqrt{33}+4\gamma},z_M(\gamma)] \quad\text{for }\gamma\in(1+\sqrt 2,3] \text{ at }P_8,
	\end{cases}
\end{align}
where $\gaSix\in(\frac{5}{3},1.71)$. Finally, we set 
\begin{equation}\label{Pg}
\begin{aligned}
	\Pg =(\Vg,\Cg)= \begin{cases} &P_5=(V_5,C_5),\quad\text{if }C_9\geq C_5, \\&P_9=(V_9,C_9),\quad\text{if }C_9< C_5.\end{cases}
\end{aligned}
\end{equation}

The results of \cite{JLS23}, concerning the analytic solution to \eqref{ODE} representing the region behind the converging shock wave, can be summarized as follows: 

\begin{theorem}[{\cite{JLS23}}]\label{collapsingthm}

(i) For all $\gamma\in (1, 3]$, there is a monotone decreasing analytic solution to \eqref{ODE} connecting $P_1=(-\frac{2}{\ga+1},\frac{\sqrt{2\ga(\ga-1)}}{\ga+1})$ to the origin. 

(ii) For $\gamma \in (1, \gaSix]$, the solution is unique (in the sense of unique $z$) and it connects $P_1$ to the origin via $P_6$. Moreover, it stays below  $B_1(V)=\sqrt{-V}$ for $V\in(V_6,0)$. The value of $z_{std}\in\mathring{\mathcal{Z}}(\gamma;P_6) $. 

(iii) For $\gamma \in (\gaSix , 2)$, there is at least one solution. If a solution connects through $P_6$, then $z_{std} \in\mathring{\mathcal{Z}}(\gamma;P_6) $, and $z_{std}$ is unique.  If it connects through $P_8$, then $z_{std} \in \mathring{\mathcal{Z}}(\gamma;P_8) $. For these ranges of $\ga$ and $z$, the solution trajectories stay below $B_1(V)=\sqrt{-V}$ for $V\in(V_8,0)$.

(iv) For $\gamma \in [2, 3]$, any such solution must connect through $P_8$ with  $z_{std}\in\mathring{\mathcal{Z}}(\gamma;P_8) $. The solution trajectories stay below $B_{\frac{3}{2}}(V)=\sqrt{-\frac{3}{2}V}$ for $V\in(V_8,0)$. 
\end{theorem}

In order to continue the solution through into the region $x>0$, we expect to find an expanding radial shock emanating from the centre of symmetry (that is, we do not look for symmetry-breaking solutions). Moreover, we seek the continuation solution again as a self-similar solution with the same parameters $\ga$ and $\la$. Thus, we search for a solution that is smooth either side of a shock curve given by $x\equiv x_H>0$. In order  to construct such a solution in the region ahead of the expanding shock, we must extend the trajectory given by Theorem~\ref{collapsingthm} for $x>0$. It emerges from this analysis that the solution trajectory extends until it meets the sonic line in the lower half-plane, $C=-(1+V)$ at a point $V>-1$.

To provide a global solution, the outgoing shock will pass the solution over the sonic line with a jump from the region $C>-(1+V)$ into the region $C<-(1+V)$. For the solution behind the reflected shock to be physically admissible, however, we require (see, for example \cite{Jenssen18}) that the trajectory after the jump connects smoothly to the point $P_\infty$ in the $(V,C)$-plane. This point is a saddle point for the ODE, and there is therefore a unique trajectory that connects to it, which we denote $(V_\infty(C),C)$, and which exists on some interval $C\in(-\infty, \Cg)$. 
 A complete solution of the expanding shock problem, therefore, is a solution trajectory extending the solution from Theorem~\ref{collapsingthm} and jumping across the lower sonic line (with the jump satisfying the Rankine-Hugoniot conditions for a shock wave) to the solution trajectory connecting to $P_\infty$, defined as a function of $x$ for $x\in[\mathring{x},\infty)$ for some $\mathring{x}>0$. 

As the flow is assumed to be self-similar with the same similarity parameters on either side of the shock, the Rankine-Hugoniot conditions may be expressed in self-similar form. Let the subscript - and + denote evaluation  ahead of and behind the reflected shock, respectively. The Rankine-Hugoniot conditions and Lax entropy condition, reformulated in the self-similar variables, are 
	\begin{align}
	1+V_+ & = \frac{\gamma-1}{\gamma+1}(1+V_-) + \frac{2C_-^2}{(\gamma+1)(1+V_-)},\nonumber\\
	C_+^2 &= C_-^2 + \frac{\gamma-1}{2}\big((1+V_-)^2-(1+V_+)^2\big),\label{jump condition}\\
	R_+(1+V_+) &= R_-(1+V_-),\nonumber\\
	C_-^2&<(1+V_-)^2.\label{supsoundspeed}
	\end{align}
Given a curve $(V(x),C(x))$ in the $(V,C)$-plane, we introduce the notation for the relevant parts of the sub- and super-sonic regions
\begin{align}\label{Ljump}
	S_U &= \{(V,C) : C\leq 0, \ 1+V>-C\},\\
	S_L& = \{(V,C) : C< 0, \ -\sqrt{\frac{\ga-1}{2\ga}}C\leq 1+V<-C\}.
\end{align}
Evidently, all points located ahead of the reflected shock belong to $S_U$.  We express the Rankine-Hugoniot conditions~\eqref{jump condition}--\eqref{supsoundspeed} by the following map
\begin{equation}\label{Id: Jp}
\begin{aligned}
\Jp: \quad S_U&\to\quad S_L\\
(V,C)&\mapsto (\Jp_1(V,C),\Jp_2(V,C))
\end{aligned}
\end{equation}
where $\Jp_1(V,C)= \frac{\gamma-1}{\gamma+1}(1+V) + \frac{2C^2}{(\gamma+1)(1+V)}-1$ and $\Jp_2(V,C) = -\sqrt{C^2 + \frac{\gamma-1}{2}\big((1+V)^2-(1+\Jp_1(V,C))^2\big)}$. 
$\Jp$ is well-defined and bijective, as shown in Lemma \ref{1-1 jump point}. We refer $\Jp(V(C),C)$ to be the jump locus of $\{(V(C),C) : C\in(a,b)\}\subsetneq S_U$ where $a<b<0$.

\begin{theorem}\label{main thm}
Let $\ga\in(1,3]$, $z=z_{std}$, and suppose that $(V(x),C(x))$ is the solution of Theorem~\ref{collapsingthm} for $x\le 0$. Then
\begin{itemize}
\item[(i)] There exists $x_H>0$ such that  $(V(x),C(x))$ extends smoothly for $x\in(0,x_H]$ as a solution of \eqref{ODE system} and there exists  {$C_H\in(-\infty,\Cg)$ such that $ \Jp(V(x_H),C(x_H))=(V_\infty(C_H),C_H)$. 
At $x_H$, the solution jumps to $(V_\infty(C_H),C_H)$ with an admissible shock and extends as a global in $x$ solution along the trajectory $(V_\infty(C),C)$.}
\item[(ii)] For $\ga\in(1,\gaSix)$, the solution is unique in the following sense. There exists a maximal smooth extension of the trajectory onto $x\in(0,x_s)$. There is a unique intersection point $P_H$ of the jump locus of the maximal solution curve with the curve $(V_\infty(C),C)$ for $C\in(-\infty,\Cg)$.
\end{itemize} 
\end{theorem}
\begin{figure}
	\centering
	\includegraphics[scale=0.25]{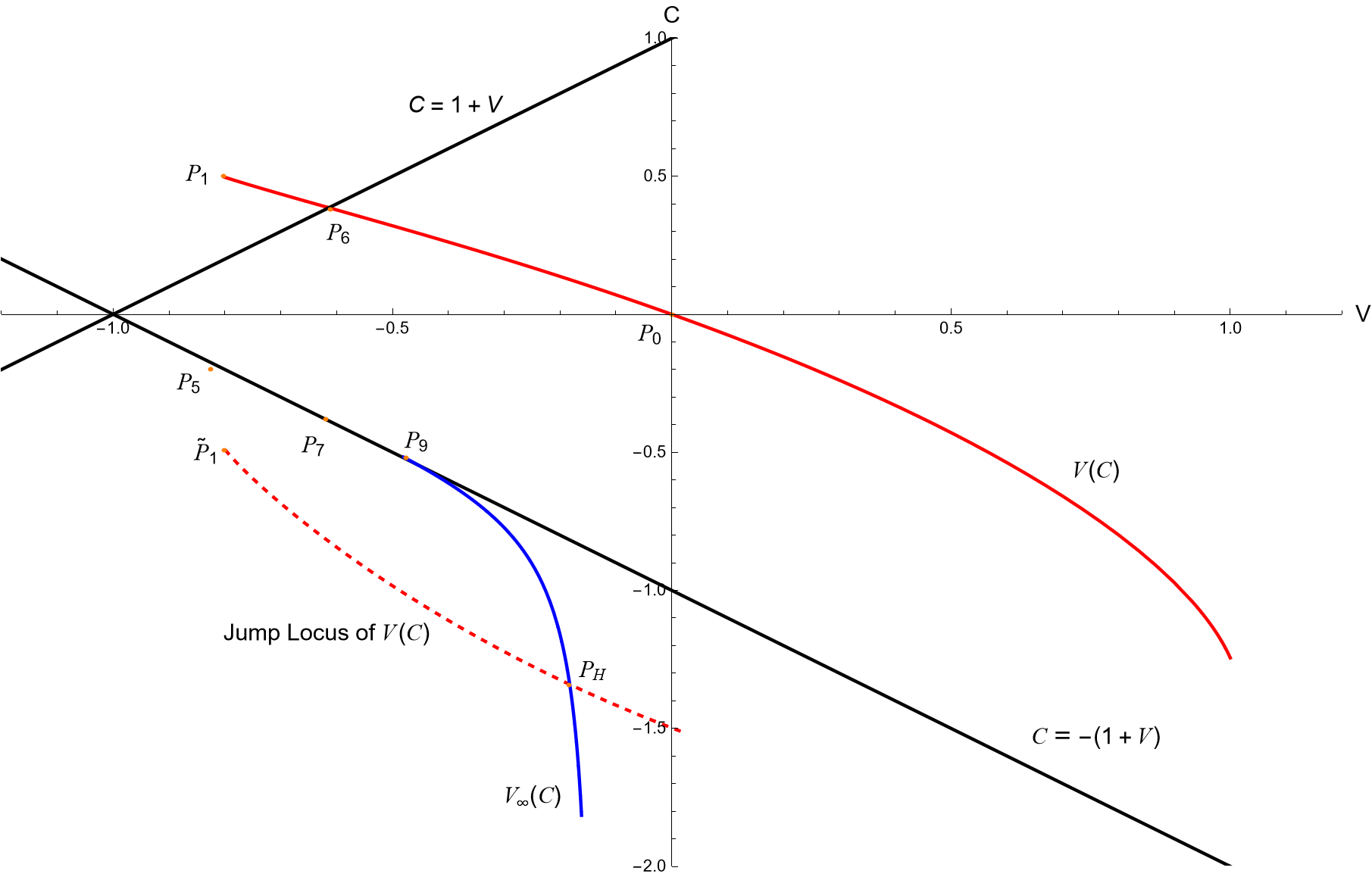}
	 \caption{the solution trajectory when $m=1$, $\gamma = 1.5$ and $z = 0.14$.}
	 \vspace{-7mm}
\end{figure}

{\begin{remark}
Following Lazarus \cite{Lazarus81}, see also \cite{Jenssen18}, it is expected that the flow extended from any other point on the jump locus of the maximal smooth solution fails to extend globally in $x$, and therefore does not give a globally defined solution of the original problem in $(t,r)$ variables. In particular, given any point
 $$(\tilde V,\tilde C)\in \mathcal{J}\big(\{(V(x),C(x))\,|\,x\in(0,x_s)\}\big)\setminus \{(V_\infty(C),C)\,|\,C\in(-\infty,\mathring{C})\},$$  the flow starting from $(\tilde V,\tilde C)$ either converges in finite $x$ to the point $\hat{P}$, at which point the density vanishes with non-physical positive pressure, or a second sonic point is reached, at which the solution loses regularity.
\end{remark}\color{black}}

{\begin{remark}
Given the solution $(V(x),C(x))$ of Theorem~\ref{main thm}, the density profile $R(x)$ is easily recovered  from the ODE \eqref{eq: ode of R}. Since $\lim_{x\to 0^-}\frac{V(x)}{x}=\lim_{x\to 0^-}\frac{dV(x)}{dx} $ is a finite number, as shown in Lemma \ref{L: finite slope in x}, $R(x)$ is well-defined and finite for $x\leq x_H$. From the jump conditions \eqref{jump condition}, we obtain the jump value of $R(x_H)$. For $x>x_H$, as $V_\infty(x)$ is increasing and $\lim_{x\to \infty} V_\infty (x)=\overline V_\infty$, a finite number, we find that $\log R(x)$ remains finite for $x>x_H$ by integrating \eqref{eq: ode of R}, and hence $R(x)$ also remains finite and positive 
for $x>x_H$. As $x\to\infty$, we find that $R(x)\to 0$ with a polynomial rate, determined by $\la$.
\end{remark}}

To prove the main theorem, we proceed in three main steps. First, we require a smooth extension of the solution through the origin in the $(V,C)$-plane. To achieve this, we show that the trajectory $C(V)$ for $x<0$ meets the origin with a finite slope. This enables us both to continue the solution into the fourth quadrant, and also to establish the smoothness of the flow with respect to the $x$ variables. The proof of this property is contained in Section~\ref{sec:origin}. We prove this through a combination of barrier arguments and a precise understanding of the Taylor expansion of any solution through the origin, a star point for the ODE~\eqref{ODE}.

Next, in Section~\ref{sec:maximal}, we establish the maximal extension of the solution trajectory, showing that the solution of~\eqref{ODE} extends smoothly in the $(V,C)$-plane until it meets the sonic line $C=-(1+V)$ at some point. It is essential for the later analysis of the jump locus that we are able to prove that this intersection occurs at a point with $C$ coordinate $C_s<\Cg$. 
The proof is delicate, and requires us to split the range of $\ga$ and $\la$ into several different regions and argue independently on each region. For some ranges of these parameters, we employ barrier arguments, while in others, we undertake an approximate integration of the ODE~\eqref{ODE} to establish precise bounds on the flow trajectory. 

The final steps, in Section~\ref{sec:jump}, are to prove the existence of the unique trajectory out of the point $P_\infty$, and to show that this trajectory meets the jump locus of the maximal extension of the solution from the origin. For this part of the proof, we require an accurate understanding of the Rankine-Hugoniot jump conditions, as well as the properties of the two solution curves.

{At various key points, especially in the barrier arguments, we must determine the sign of certain polynomial functions on given intervals. The Fourier-Budan theorem, a classical result used to determine the number of zeros of a polynomial within a certain range, is employed to prove these sign conditions.}

We conclude this section with some basic properties of $F(V,C)$ and $G(V,C)$ in the regions $S_U\cap\{V\geq V_8\}$ and $S_L\cap\{V\geq V_8\}$, which will play a crucial role in understanding the dynamics of the system \eqref{ODE system}.

\begin{lemma}\label{L:FG=0}
Let $m=1,2$. For any $\ga\in(1,3]$ and $z\in(0,z_M]$. In regions $\mathcal S_U\cap\{V\geq V_8\}$ and  $\mathcal S_L\cap\{V\geq V_8\}$,   the followings hold for functions $F(V,C)$ and $G(V,C)$:
\begin{enumerate}
	\item[(i)] $F(V,C)> 0$ in the region $\mathcal S_U\cap\{V\geq V_8\}$. $F(V,C)=0$ has a unique simple root branch $V=V_F^+(C)$ in the region $\mathcal S_L\cap\{V\geq V_8\}$ with the domain $C\in[-\sqrt\frac{1+m\ga z}{1+m z},C_9)$. Moreover, $V=V_F^+(C)$ connects $P_9$ and the point $(0,-\sqrt\frac{1+m\ga z}{1+m z})$ with $(V_F^+)'(C)<0$ for $C\in[-\sqrt\frac{1+m\ga z}{1+m z},C_9)$, and $\lim_{C\to -\infty}V_F^+(C) = +\infty$.
	\item[(ii)] $G(V,C)=0$ has a unique simple root branch, denoted by $V_G^+(C)$, in $\mathcal S_U\cap\{V\geq V_8\}$. $V_G^+(C)$ satisfies that $V_G^+(0)=0$, $\lim_{C\to -\infty}V_G^+(C) = +\infty$ and $(V_G^+)'(C)<0$ for $C<0$. In $\mathcal S_L\cap\{V\geq V_8\}$, $G(V,C)=0$ has another unique simple root branch, denoted by $V_G(C)$ which satisfies $V_8< V_G(C)$, $\lim_{C\to -\infty}V_G(C) = \bar V_{\infty}$ and $V_G'(C)<0$ for $C< 0$.
\end{enumerate} 
\end{lemma}
\begin{proof}
The proof is given in Appendix \ref{Ap: FG=0}.
\end{proof}

\section{Continuation through the origin}\label{sec:origin}
In this section, we first  establish that the solution $(V(x),C(x))$ passes through $(0,0)$ with finite and nonzero slope, thus passing into the fourth quadrant from the second quadrant in the  $(V,C)$-plane. Our strategy is to use the asymptotic behaviour of $C(V)$ around $V=0$ to obtain our conclusion. We begin by showing $\limsup_{V\to 0^-}\frac{C(V)}{V}\neq 0$. Recall that the solution trajectory is obtained as the solution to the ODE
\begin{align}\label{ivp p*}
		\begin{cases}
			&\cfrac{dC}{dV} = \cfrac{F(V,C;\gamma,z)}{G(V,C;\gamma,z)},\quad V\in[V_*,0),\\
			&C(V_*)= C_*, C'(V_*) = c_1.
		\end{cases}
	\end{align}
	For notational convenience, we define
\begin{align}
	\Gamma(P_*) = \begin{cases}
	(1,2] &\text{ when }P_*=P_6,\\
		(\gaSix,3] &\text{ when }P_*=P_8. 
		\end{cases}
\end{align}
		We further define
\begin{align}
	s(\ga,z;P_*) : = \frac{(\ga-1)C_*(\ga,z)}{4V_*(\ga,z)}\label{s(ga,z)} <0 . 
\end{align}

\begin{proposition}\label{lower bound of v1}
	Let $P_*$ be $P_6$ or $P_8$. For any $\ga\in\Gamma(P_*) $ and $z\in \mathring{\mathcal{Z}}(\ga;P_*)$, the solution $C(V)$ to  \eqref{ivp p*} satisfies  $\limsup_{V\to 0^-}C'(V)\leq s(\gamma,z;P_*)$. Equivalently, as $C(V)$ is strictly monotone, $\liminf_{C\to 0^+}V'(C)\geq \frac{1}{s(\gamma,z;P_*)}$.
\end{proposition}

\begin{proof}
We prove the Proposition by showing that the curve $C = s(\ga,z)V$ is a lower barrier for the solution to \eqref{ivp p*}, namely $C(V) >  s(\ga,z)V$ for $V\in (V_*,0)$. As $C\to 0^+$ as $V\to 0^-$ by Theorem \ref{collapsingthm}, we then deduce  $\limsup_{V\to 0^-}C'(V)\leq s(\gamma,z;P_*)$ as required. 

	For convenience, we write $s=s(\ga,z;P_*)$. To prove the lower barrier holds,
	as it is clear that $sV_*=\frac{(\ga-1)C_*}{4}<C_*$ for any $\ga\in(1,3]$, 
	it suffices to show that, for $P_*\in\{P_6,P_8\}$, $m=1,2$,  $\gamma\in(1,3]$ and $z\in\mathring{\mathcal{Z}}(\ga;P_*)$, we have
	\begin{align*}
		\frac{d}{dV} (C - s V) \Big | _{(V, sV)}=\frac{F(V,sV;\gamma,z)}{G(V,sV;\gamma,z)}-s>0, \ \text{for } V \in(V_*,0). 
	\end{align*}
	Since $G(V,sV;\ga,z)>0$ for $V\in(V_*,0)$, it is enough to show that 
	\begin{align*}
		F(V,sV;\gamma,z)-sG(V,sV;\gamma,z)>0.
	\end{align*}
	From \eqref{G(V,C)} and \eqref{F(V,C)}, we compute 
	\begin{align*}
		&F(V,sV;\gamma,z)-sG(V,sV;\gamma,z) = \frac{msV^2}{2} \Barrier_s (V,z)  ,
	\end{align*}
 	where  
	\begin{align}
	\Barrier_s (V,z) : = -[2s^2+\ga-1]V+\frac{2zs^2V}{1+V}-[4zs^2+(1-\gamma z)(\gamma-1)].
	\end{align}
	Since $msV^2<0$, it is sufficient to prove that
$		\Barrier_s (V,z)<0.$
	It is clear that $\Barrier_s (0,z) = -[4zs^2+(1-\ga z)(\ga-1)]<0$, and $\Barrier_s (V,z)$ is a concave function in the interval $V\in[V_*,0]$ as $V_*>1$. It is therefore enough to prove that $\Barrier_s (V_*,z) <0$ and  $\partial_V\Barrier_s (V_*,z)<0$. 
	We will treat $P_*=P_8$ and $P_*=P_6$ separately. When $P_*=P_8$, a direct computation gives
	\begin{align*}
		\Barrier_s (V_8,z)  &= -\frac{(\ga-1)(\ga-5)(\ga+3)z^2}{16(1+V_8)V_8^2}\Big[(\ga-2)^3z^2-(\ga-2)(\ga+4)z-2+[(\ga-2)^2z-2]w(z)\Big]<0
	\end{align*}
	where we used $z<z_M<\frac{1}{5}$ and $\ga\in(1,3]$ in the last inequality. Moreover, we see that 
	\begin{align*}
	\partial_V\Barrier_s (V_8,z) = -2s^2(\ga,z;P_8)-(\ga-1)+\frac{2zs^2(\ga,z;P_8)}{(1+V_8)^2} <  \frac{2s^2}{(1+V_8)^2}(z-(1+V_8)^2), 
	\end{align*}
	where we have used $\ga>1$. By  \eqref{ineq: V'6>0, V'8<0}, $z-(1+V_8)^2$ is increasing in $z\in(0,z_M]$, and hence, by  \eqref{Id: VCzM},
	\begin{align*}
		z-(1+V_8)^2\leq z_M-(1+V_8(z_M))^2 = (1-\ga)z_M<0.
	\end{align*}
	This completes the proof of $P_*=P_8$. 
	 In the case of $P_*=P_6$, from Theorem \ref{collapsingthm}, we only need to check $\Barrier_s (V_6,z)<0$ and $\partial_V\Barrier_s (V_6,z)<0$ for $\ga\in(1,2)$. We compute
	 \begin{align*}
	 		\Barrier_s (V_6,z)  &=- \frac{(\ga-1)(\ga-5)(\ga+3)z^2}{16(1+V_6)V_6^2}\Big[(\ga-2)^3z^2-(\ga-2)(\ga+4)z-2-[(\ga-2)^2z-2]w(z)\Big]\\
			&=-\frac{\ga^2(\ga-1)(\ga-5)(\ga+3)z^3}{2(1+V_6)V_6^2[(\ga-2)^3z^2-(\ga-2)(\ga+4)z-2+[(\ga-2)^2z-2]w(z)]}<0, 
	 \end{align*}
	where we have used $C_6=1+V_6$ in \eqref{s(ga,z)}, $\ga\in(1,2)$ and $z<z_M<\frac{1}{5}$. Recalling \eqref{P6} and \eqref{w(z)}, 
	\begin{align*}
		\partial_V\Barrier_s (V_6,z) = -2s^2(\ga,z;P_6)-(\ga-1)+\frac{2zs^2(\ga,z;P_6)}{(1+V_6)^2} =(1-\ga)\Big[1-\frac{\ga-1}{8V_6^2}(z-(1+V_6)^2)\Big]<0,
	\end{align*}
	where we used $V_6\leq V_6(z_M(\ga))< V_6(z_M(2))=-\frac{1}{2}$, $\ga\in(1,2)$ and $z\leq z_M<\frac{1}{5}$ in the last inequality. This finishes the proof.
\end{proof}

In order to show that $C(V)$ passes through the origin with finite slope, it remains to establish, firstly, that the trajectory is at least $C^1$ through the origin, and also to eliminate the possibility that $\lim_{V\to 0^-} C'(V)=-\infty$, equivalently, $\lim_{C\to0^+}V'(C)=0$. 
As the origin is a star point for \eqref{ODE}, the trajectory can leave, in principle, at any slope. We show that, in fact, there exists an analytic trajectory $V(C)$ with non-zero slope at the origin that coincides with the solution to \eqref{ivp p*}. 
To construct this analytic solution, we first make the formal Taylor expansion 
\begin{align}\label{eq:V(C)Taylor}
	V(C) = \sum_{\ell=1}^{\infty} v_\ell C^\ell.
\end{align}
We denote, from the formal Taylor expansion,
\begin{align}
	V'(C)F(V,C)(1+V) - G(V,C)(1+V) = \sum_{\ell=1} \mathcal{V}_\ell C^\ell.
\end{align}
For an analytic solution of \eqref{ODE}, clearly $\mathcal{V}_\ell=0$ for all $\ell\geq 1$.
For convenience, we employ the notation
\begin{align}
	V'(C) = \sum_{\ell=1}^{\infty} \ell v_\ell C^{\ell-1},\quad (v^2)_\ell = \sum_{i+j = \ell}v_iv_j, \quad (v^3)_\ell = \sum_{i+j+k = \ell}v_iv_jv_k,\quad  (v^4)_\ell = \sum_{i+j+k+l = \ell}v_iv_jv_kv_l.\notag
\end{align}
With these conventions, we obtain, from inserting \eqref{eq:V(C)Taylor} into \eqref{F(V,C)}
\begin{align}\label{V'(C)F(V,C)(1+V)}
	V'(C)&\,F(V,C)(1+V)\notag\\
	& = V'(C)C[C^2(1+V+mz)-a_1V^3-(3a_1-a_2)V^2-(1+m\gamma z+2a_1-a_2)V-(1+m\gamma z)]\notag\\
	&=-\sum_{\ell=1}^{\infty}a_1 \frac{\ell}{4}(v^4)_\ell C^\ell- \sum_{\ell=1}^{\infty} (3a_1-a_2)\frac{\ell}{3}(v^3)_\ell C^\ell- \sum_{\ell=1}^{\infty} (1+m\gamma z+2a_1-a_2)\frac{\ell}{2}(v^2)_\ell C^\ell\notag\\
	&\quad+\sum_{\ell=1}^{\infty}\frac{\ell}{2}(v^2)_\ell C^{\ell+2}+ \sum_{\ell=1}^{\infty} (1+mz)\ell v_\ell C^{\ell+2}-\sum_{\ell=1}^{\infty} (1+m\gamma z)  \ell v_\ell C^\ell.
\end{align}
Moreover, from \eqref{G(V,C)},
\begin{align}\label{G(V,C)(1+V)}
	G&\,(V,C)(1+V)\notag\\
	& =-V^4-(3+m\ga z)V^3-(3+2m\ga z)V^2-(1+m\ga z)V+(m+1)C^2V^2+(m+1+2mz)C^2V+2mzC^2\notag \\
	&=- \sum_{\ell=1}^{\infty} (v^4)_\ell C^\ell- \sum_{\ell=1}^{\infty} (3+m\ga z)(v^3)_\ell C^\ell-\sum_{\ell=3}^{\infty} (3+2m\ga z)(v^2)_\ell C^\ell+\sum_{\ell=1}^{\infty}(m+1) (v^2)_\ell C^{\ell+2}\notag\\
	&\quad-\sum_{\ell=1}^{\infty} (1+m\ga z)v_\ell C^\ell+\sum_{\ell=1}^{\infty}(m+1+2mz) v_\ell C^{\ell+2}+2mzC^2.
\end{align}
Now we study the difference of \eqref{V'(C)F(V,C)(1+V)} and \eqref{G(V,C)(1+V)}. First of all, at order 1, we have 
\begin{align*}
	\mathcal{V}_1=-(1+m\ga z)v_1+(1+m\ga z)v_1 = 0.
\end{align*}
Therefore, $v_1$ can take any value, as expected since $P_0$ is a stable node. By Proposition \ref{lower bound of v1}, we may suppose $v_1\in[\frac{1}{2s(\ga,z;P_*)},0]$. 
At order 2, we obtain
\begin{align*}
	\mathcal{V}_2&=- (1+m\gamma z+2a_1-a_2)v_1^2-2(1+m\ga z)v_2+(3+2m\ga z)v_1^2+(1+m\ga z)v_2-2mz \\
	&=-(1+m\ga z)v_2- \frac{m(\ga-1)(1-\ga z)}{2}v_1^2-2mz.
\end{align*}
This implies, for a solution of \eqref{ODE},
\begin{align}\label{v2}
	v_2 = \frac{-2mz- \frac{m(\ga-1)(1-\ga z)}{2}v_1^2}{1+m\ga z}<0.
\end{align}
For $\ell\geq 3$, recalling the convention that $v_\ell=0$ if $\ell\leq 0$, we obtain the coefficient
\begin{align*}
	\mathcal{V}_\ell = 
	-A_\ell v_\ell+B_\ell,
\end{align*}
where 
\begin{align}
	A_\ell&=(1+m\ga z)(\ell-1)> 0,\label{def:Aell}\\
	B_\ell & =(1-\frac{\ell a_1}{4})(v^4)_\ell + [3+m\ga z-\frac{\ell(3a_1-a_2)}{3}](v^3)_\ell+[3+2m\ga z-\frac{\ell(1+m\ga z + 2a_1-a_2)}{2}](v^2)_\ell\notag \\
	&\quad+ (\frac{\ell-2}{2}-1-m)(v^2)_{\ell-2}+[(1+mz)(\ell-2)-m-1-2mz]v_{\ell-2}.\label{def:Bell}
\end{align}
We notice that $B_\ell = B(\ell,v_1,v_2,..., v_{\ell-1})$. Thus, to have $\mathcal{V}_\ell=0$, we require
\begin{align}
	v_\ell = \frac{B_\ell}{A_\ell}. \label{v_l equation}
\end{align}
In the following, we prove that the solution to the recurrence relation \eqref{v_l equation} yields a convergent Taylor series, and hence an analytic solution of \eqref{ODE}. First, we prove a growth rate bound in $\ell$ for $v_\ell$. 
\begin{lemma}\label{lemma for B_L upperbound}
Let $P_*$ be $P_6$ or $P_8$. For any fixed $\gamma\in\Gamma(P_*) $ and $z\in\mathring{\mathcal{Z}}(\ga;P_*)$, let $\alpha\in (1,2)$ be given and suppose $v_1\in[\frac{1}{2s(\ga,z;P_*)},0]$. Then, there exists a constant $K_*=K_*(\gamma)>1$ such that if $K\geq K_*$ and $L\geq 5$, then if also the following inductive assumption holds,
\begin{align}\label{v_l assumption}
    |v_{\ell}| \leq \cfrac{K^{\ell-\alpha}}{\ell^3},\quad 2\leq \ell \leq L-1,
\end{align}
then we have
\begin{align}\label{B_L upperbound}
    |B_L| \leq \beta \cfrac{K^{L-\alpha}}{L^2} \Big(
    \cfrac{1}{K^{\alpha-1}}+\frac{1}{K}\Big)
\end{align}
for some constant $\beta=\beta(\gamma,\lambda)$.
\end{lemma}

For the proof, we will require the following result from \cite{GHJS22} to estimate certain combinations of coefficients. 
\begin{lemma}[{\cite[Lemma B.1]{GHJS22}}]\label{lem:B1}
There exists a universal constant $a>0$ such that for all $L\in \mathbb{N}$, the following inequalities hold
\begin{align*}
    \sum_{\substack{i+j+k = L\\ i,j,k\geq 1}}\cfrac{1}{i^3j^3k^3} \leq \cfrac{a}{L^3}\quad\text{and}\quad
    \sum_{\substack{i+j = L\\ i,j\geq 1}} \cfrac{1}{i^3j^3} \leq \cfrac{a}{L^3}.
\end{align*}
\end{lemma}
A similar argument also yields
\begin{align*}
    \sum_{\substack{i+j+k+l = L\\ i,j,k,l\geq 1}}\cfrac{1}{i^3j^3k^3l^3}& \leq \cfrac{a}{L^3}.
\end{align*}
\begin{proof}[Proof of Lemma \ref{lemma for B_L upperbound}]
 From \eqref{v_l equation}, we may write $v_\ell = \frac{B_\ell(\ell,v_1,\ga,z)}{A_\ell(\ga,z)}$. It is trivial to see there exists a finite $K_* = K_*(\ga)>1$ such that \eqref{v_l assumption} holds for $\ell=1,2,3,4$.

First, by using the induction assumption \eqref{v_l assumption}, $v_\ell=0$ if $\ell\leq 0$ and Lemma \ref{lem:B1}, we have
\begin{align*}
	|(v^4)_L| &
	= \Big|  \sum_{\substack{i+j+k+l = L \\ i,j,k,l\geq 1}}v_iv_jv_kv_l \Big| \leq K^{L-4\alpha}   \sum_{\substack{i+j+k+l = L \\ i,j,k,l\geq 1}}\frac{1}{i^3j^3k^3l^3}\lesssim \frac{K^{L-4\alpha} }{L^3} ,\\
	|(v^3)_L| &
	= \Big|  \sum_{\substack{i+j+k = L \\ i,j,k\geq 1}}v_iv_jv_k \Big| \leq K^{L-3\alpha}   \sum_{\substack{i+j+k = L \\ i,j,k\geq 1}}\frac{1}{i^3j^3k^3}  \lesssim \frac{K^{L-3\alpha} }{L^3},\\
	|(v^2)_L| &
	= \Big|  \sum_{\substack{i+j = L \\ i,j,k\geq 1}}v_iv_j \Big| \leq K^{L-2\alpha}   \sum_{\substack{i+j = L \\ i,j\geq 1}}\frac{1}{i^3j^3} \lesssim \frac{K^{L-2\alpha} }{L^3},\\
		|(v^2)_{L-2}| &
		= \Big|  \sum_{\substack{i+j = L-2 \\ i,j,k\geq 1}}v_iv_j \Big| \leq K^{L-2-2\alpha}   \sum_{\substack{i+j = L \\ i,j\geq 1}}\frac{1}{i^3j^3} \lesssim \frac{K^{L-2-2\alpha} }{L^3},
\end{align*}
where we have used  
the assumptions $\al\in(1,2)$ and $K\geq 1$ and, moreover, $L-2> 2$.
Note that as $L\geq 5$, there exists a universal constant $M>0$ such that $\frac{L}{L-2}\leq M$.

Now we  estimate $B_L$, recalling the definition in \eqref{def:Bell}, by employing these combinatorial estimates as
\begin{align*}
    |B_L| & \lesssim L\bigg(\Big|(v^4)_L\Big| + \Big|(v^3)_L\Big| + \Big|(v^2)_L\Big|+\Big|(v^2)_{L-2}\Big| +|v_{\ell-2}|\bigg) \\
    & \lesssim (L+1)\bigg(\cfrac{K^{L-4\alpha}}{L^3} + \cfrac{K^{L-3\alpha}}{L^3} + \cfrac{K^{L-2\alpha}}{L^3} + \cfrac{K^{L-2-2\alpha}}{L^3} + \cfrac{K^{L-2-\alpha}}{L^3}\bigg) \lesssim \cfrac{K^{L-\alpha}}{L^2} \Big(\cfrac{1}{K^{\alpha-1}}+\frac{1}{K}\Big).
\end{align*}
This completes the proof.
\end{proof}

We next justify the inductive growth assumption on $v_\ell$. 

\begin{lemma}\label{lemma for c_l}
Let $P_*$ be $P_6$ or $P_8$. For any fixed $\gamma\in\Gamma(P_*) $ and $z\in\mathring{\mathcal{Z}}(\ga;P_*)$,  let $\alpha\in(1,2)$ be given. Suppose $v_1\in[\frac{1}{2s(\ga,z;P_*)},0]$ and let $v_{\ell}$ solve the recursive relation \eqref{v_l equation} for $\ell\geq 2$. Then there exists a constant $K=K(\ga,\la)>1$ such that 
$c_\ell$ satisfies the bound
\begin{align}\label{c_l condition}
    |v_{\ell}| \leq \cfrac{K^{\ell-\alpha}}{\ell^3}. 
\end{align}
\end{lemma}
\begin{proof} We argue by induction on $\ell$. 
When $\ell =2,3,4$, it is clear from \eqref{c_l condition}, the forms of $A_\ell$ and $B_\ell$ defined by \eqref{def:Aell}--\eqref{def:Bell} that there exists a constant $K(\gamma,\lambda)$ such that $v_2$, $v_3$, and $v_4$ satisfy the bounds. 
Suppose for some $L\geq 5$, \eqref{c_l condition} holds for all $2\leq \ell \leq L-1$. Then we may apply Lemma \ref{lemma for B_L upperbound}  
and with the recursive relation \eqref{v_l equation}, we obtain 
\begin{align*}
    |v_L| \leq \cfrac{\beta}{|A_L|} \cfrac{K^{L-\alpha}}{L^2} \Big(\frac{1}{K^{\alpha-1}}+\frac{1}{K}\Big),
\end{align*}
where $A_L = (1+m\ga z)(L-1)\geq \frac L2$ for $L\geq 5$. 
Therefore,
\begin{align*}
    |v_L| \leq \frac{2\beta}{ L} \cfrac{K^{\ell-\alpha}}{L^2} \Big(\frac{1}{K^{\alpha-1}}+\frac{1}{K}\Big). 
\end{align*}
Choosing $K$  sufficiently large, as $\al>1$, it is clear that the estimate \eqref{c_l condition} holds for $\ell = L$, thus concluding the proof.
\end{proof}

We therefore obtain the following theorem by standard arguments.
\begin{theorem}\label{analytic}
For any fixed $\gamma\in(1,3]$, $z\in\mathring{\mathcal{Z}}$ and $v_1\in[\frac{1}{2s(\gamma,z;P_*)},0]$ where $s(\ga,z;P_*)$ is defined by \eqref{s(ga,z)}, there exists $\epsilon=\epsilon(\ga,z)>0$ such that the Taylor series
\begin{align}\label{local analytic solution}
	V(C) = \sum_{\ell =1}^{\infty} v_{\ell}C^{\ell}
\end{align}
 converges absolutely on the interval $(-\epsilon,\epsilon)$. Moreover, $V(C)$ is the unique analytic solution to \eqref{ODE} with the given choice of $v_1$.
\end{theorem}

\begin{lemma}\label{pass collapse}
Let $P_*$ be $P_6$ or $P_8$, and $m=1,2$. For any fixed $\gamma\in\Gamma(P_*) $ and $z\in\mathring{\mathcal{Z}}(\ga;P_*)$,  the solution to the system \eqref{ivp p*} analytically passes through $P_0$ with a nonzero and finite slope (i.e. $-\infty<v_1<0$).
\end{lemma}

\begin{proof}
	Fix $\gamma\in\Gamma(P_*) $ and $z\in\mathring{\mathcal{Z}}(\ga;P_*)$,  let $\epsilon$ be given as  in Theorem \ref{analytic}, and denote by $V(C)$ the solution of \eqref{ivp p*}. 
We observe that the local analytic solution $\overline V(C)$ at $P_0$, obtained in Theorem \ref{analytic}, depends on $v_1$, $\gamma$ and $z$ within the domain $C\in(0, \epsilon)$. Through standard compactness and uniform convergence arguments, the coefficients $v_\ell = v_\ell(v_1)$ are continuous functions of $v_1$ for each $\ell\geq 2$ and $v_1\in[\frac{1}{2s(\ga,z;P_*)},0]$. Consequently, we deduce that $\overline V(C)=\overline V(C;v_1)$ is a continuous function of $v_1$ on its domain as well. From Proposition \ref{lower bound of v1}, there exists $\tilde\epsilon\in(0,\epsilon)$ such that $\overline{V}(C;\frac{1}{2s(\ga,z;P_*)})< V(C)$ for $C\in(0,\tilde \epsilon)$. 

Moreover, from \eqref{v2},  we have $v_2|_{v_1=0} = -\frac{2mz}{1+m\ga z}$. As $\frac{d}{dz} \big(v_2|_{v_1=0}\big) = - \frac{2m}{(1+m\ga z)^2}<0$, we deduce
\begin{align}
	v_2|_{v_1=0} \geq -\frac{2mz_M}{1+m\ga z_M}  = -\frac{2m}{2+\ga+2\sqrt{2\ga}+m\ga} > -\frac{2}{3}.
\end{align}	
Thus, in the interval $C\in(0, \epsilon)$ (possibly shrinking $\epsilon$), $\overline{V}(C;0)$ is bounded below by $V(C) = -C^2$ and $V(C) = -\frac{2}{3}C^2$, which are the upper barriers of the solution curves as stated in Theorem \ref{collapsingthm}. Consequently, the solution curve $V(C)< \overline{V}(C;0)$ for $C\in(0,\epsilon)$. By a simple continuity argument, there  exists $v_1\in ( \frac{1}{2s(\ga,z;P_*)},0)$ such that $\overline{V}(C;v_1)$ intersects the solution curve at some $C\in(0, \epsilon)$. The uniqueness of  solutions to the ODE implies that these two curves coincide within a neighborhood of $P_0$ and, in particular, pass through $P_0$ with finite slope.
\end{proof}

\begin{lemma}\label{L: finite slope in x}
In the original $x$ coordinates, the solution $(V(x),C(x))$ is $C^1$ with finite slope through $x=0$.
\end{lemma}

\begin{proof}
Recall from \cite[Remark 2.10]{JLS23} that the solution $(V(x),C(x))\to(0,0)$ as $x\to 0$. Moreover, from Lemmas~\ref{analytic}--\ref{pass collapse}, we can write $C(V) = c_1 V+O(V^2)$. Thus, from the explicit forms of $G$ and $D$ in~\eqref{D(V,C)}--\eqref{G(V,C)}, we see that, for $x$ near 0, 
\begin{equation}\label{exp:GDnearx=0}
G(x)=C(x)^2 g_1(x) - g_2(x) = -\lambda V(x) (1 + O(V(x))),\qquad D=1+O(V(x)).
\end{equation}
So then, recalling that as $x\to0^-$, $V(x)\to 0$ from below monotonically, we have
\begin{equation*}
V'(x)=-\frac{1}{\lambda x}\frac{G(x)}{D(x)} = \frac{V(x)(1+O(V))}{x},
\end{equation*}
and hence, given $\epsilon\in(0,\frac14)$, for $x<0$ sufficiently small,
$$\frac{(1-\epsilon)V(x)}{x}\leq V'(x)\leq \frac{(1+\epsilon)V(x)}{x},$$
which yields, after integration,  bounds on $V$ of
$$\frac1C(-x)^{1+\epsilon} \leq |V(x)| \leq C(-x)^{1-\epsilon}.$$
Returning to~\eqref{exp:GDnearx=0}, we insert this back into the expansions of $G$ and $D$, and find
$$V'(x)=V(x)/x +O(x^{1-2\epsilon}),$$
 which can be integrated to obtain a limit for $\frac{V(x)}{x}$, and hence to deduce that $V\in C^1$ up to $x=0$. It follows then from the analytic form $C(x)=C(V(x))$ that $C$ is also $C^1$.
\end{proof}

\section{Maximal extension}\label{sec:maximal}

The goal of this section is to analyse the maximal smooth extension of the flow through the origin. In particular, we will show first that the trajectory continues monotonically in the fourth quadrant (that is, the region $\{(V,C)\mid V>0, C<0\}$) until meeting the curve $V=V_G^+(C)$, or equivalently $C = C_G^+(V)$. It then extends further until intersecting the lower sonic line $\{C=-(1+V)\}$ in either the  fourth or third quadrants ($\{(V,C)\mid V<0, C<0\}$). It is essential for the later analysis of Section~\ref{sec:jump} that we show that this intersection occurs at a point $(V_s,C_s)$ with $C_s<\Cg$. In particular, the flow cannot continue smoothly through $(V_s,C_s)$, and thus we  obtain the maximal extension of the smooth solution.

Our first preparatory  lemma ensures intersection of the solution curve and the curve $V=V_G^+(C)$. This follows by a contradiction argument and the use of a suitable upper bound of the solution curve.

\begin{lemma}\label{intersect G=0}
Let $P_*$ be $P_6$ or $P_8$. For any fixed $\gamma\in\Gamma(P_*) $ and $z\in\mathring{\mathcal{Z}}(\ga;P_*)$,  the maximal smooth extension of the solution given by Lemma \ref{pass collapse} must intersect $G=0$ at some point in the fourth quadrant. As a consequence, there exists $C_s(\ga,z)<0$ such that the solution curve approaches the sonic line $C=-V-1$ at $(-C_s-1,C_s)$.
\end{lemma}

\begin{proof}
We argue by contradiction.
Suppose the solution trajectory does not meet with the curve  $C = C_G^+(V)$, and so the solution trajectory $C(V)$ remains above the curve  $C = C_G^+(V)$ for all $V>0$. By Lemma \ref{L:FG=0}, we therefore have $G(V,C)<0$, $F(V,C)>0$ along the solution trajectory, so that also $C'(V)<0$.  Recalling \cite[Lemma 5.1]{JLS23}, we compute
\begin{align*}
    \frac{C^2f_1-f_2}{C^2g_1-g_2}-\frac{f_2}{g_2} =\cfrac{C^2(f_1g_2-f_2g_1)}{(C^2g_1-g_2)g_2}
    =-\frac{m[(m + 1)(\gamma- 1) + 2]}{2}\frac{(V-V_4)(V-V_6)(V-V_8)C^2}{(C^2g_1-g_2)g_2} >0
\end{align*}
where we used $\ga>1$, $G(V,C)<0$, $g_2(V)>0$ and $V>0>V_4,V_6,V_8$ in the last inequality.

Thus, given $V_0>0$, for $V>V_0$, we have 
\begin{align}
 \ln (-C(V)) = \int_{V_0}^V \cfrac{C^2f_1-f_2}{C^2g_1-g_2} dV +\ln (-C(V_0)) > \int_{V_0}^V \cfrac{f_2}{g_2} dV +\ln (-C(V_0)) =: \psi(V).
\end{align}
Therefore, we obtain a upper bound for the solution $C(V)$ in the interval $V\in(V_0,+\infty)$ of
\begin{align*}
C(V) < -e^{\psi (V)}=:\Psi(V).
\end{align*}
A direct computation gives
\begin{align}
	\Psi(V) = -MV(1+V)^{\frac{1-\ga}{2}}(1+m\ga z+V)^{\frac{(m+1)(\ga-1)}{2}},
\end{align}
where $M = |C(V_0)|V_0^{-1}(1+V_0)^{\frac{\ga-1}{2}}(1+m\ga z+V_0)^{\frac{(m+1)(1- \ga)}{2}}$ is a positive constant. On the other hand, we notice that the curve $C = C_G^+(V)$  in the 4th quadrant is explicitly given by
 \begin{align}
 	\left\{C=C_G^+(V) = -\sqrt{\frac{V(1+V)(1+m\ga z+V)}{(m+1)V+2mz}}\right\}.
 \end{align}
 For $V$ large enough, we observe that $C_G^+(V)$ behaves as a linear function in $V$, while $\Psi(V)$ has superlinear growth. Thus, $\Psi(V)$ intersects $C = C_G^+(V)$ for some $V>V_0$, a contradiction as $\Psi(V)$ is an upper bound of the solution curve $C(V)$.
 
 To show the second claim of the Lemma, note first that, by Lemma~\ref{L:FG=0}, after the trajectory passes through the curve $\{C=C_G^+(V)\}$, it can be parametrised as $(V(C),C)$, with $V'(C)>0$. Thus, there are no stationary points of the flow between the lower sonic line and $\{C=C_G^+(V)\}$, and so it is clear that as $C$ decreases, the flow must converge to the lower sonic line at some point $(-C_s-1,C_s)$. 
\end{proof}

In the subsequent subsections, we perform a separate analysis of the solution trajectories originating from $P_6$ and $P_8$, representing them in the form  $V(C)$. Our specific objective is to establish that these solution trajectories are bounded below (as functions $V(C)$) by precise lower barriers within the interval $C\in(0,\Cg)$ and therefore to justify that $C_s<\Cg$.

\subsection{Continuation of trajectories from $P_8$}
We begin by showing $C_s<C_9$ in the case  $P_*=P_8$.
\begin{lemma}\label{barrier P8}
For any $\ga\in \Gamma(P_8)$ and $z\in \mathring{\mathcal{Z}}(\ga;P_8)$, the solution, as guaranteed by Lemma \ref{intersect G=0}, approaches the sonic line $C=-V-1$ at some point $(-C_s-1,C_s)$ with  $C_s<C_9$.
\end{lemma}

\begin{proof}
The existence of $C_s$ follows from Lemma~\ref{intersect G=0}. Suppose that $C_s\geq C_9$. Then, for $\ga\in(\gaSix,2)$, the solution must cross the curve  $V(C) = -C^2$, while for $\ga\in[2,3]$, it must cross $V(C) = -\frac{2}{3}C^2$. In either case, the crossing must occur with $C\geq C_9$. Then from Theorem \ref{collapsingthm} and the symmetry of the system around the $V$-axis, we arrive at a contradiction. 
\end{proof}

Recalling that we require the maximal smooth extension to meet the sonic line below $\mathring{C}$, it is necessary to extend the validation of the lower barrier $V(C) = -\frac{2}{3}C^2$ from the interval $C\in(C_9,0)$ to $C\in(C_5,0)$ for $\ga\in[\ga_g,3]$, where $\Cg = C_5$ and $z\in[z_g,z_M]$ (compare \eqref{P4}--\eqref{zg}).  Since $V(C) = -\frac{2}{3}C^2$ intersects $C=-(1+V)$ at $(\frac{\sqrt {33}-7}{4},-\frac{\sqrt {33}-3}{4})$, we initiate this process by establishing that $C_5>-\frac{2}{3}>-\frac{\sqrt {33}-3}{4}$.

\begin{lemma}\label{C4<2/3}
Let $m=1,2$. For any $\ga\in[\ga_g,3]$ and $z\in[z_g,z_M]$,
\begin{align*}
	C_5(z)>-\frac{2}{3}.
\end{align*}
\end{lemma}

\begin{proof}
Recalling \eqref{P4}, we have
$V_4 = \frac{-2m\ga z-2}{(m+1)\ga+1-m}<\frac{-2m z}{m+1-\frac{m-1}{\ga}}<-\frac{2m z}{m+1}$. Hence, again from \eqref{P4}, we see that $C_5^2<\frac49$ is  	 equivalent to showing
	\begin{align*}
		 p(V_4):=9V_4^3+9(2+m\ga z)V_4^2+[5+m(9\ga z-4)]V_4-8mz >0.
		\end{align*}
	We first show that $p(V)$ only has one zero in the interval $V\in[-1,0]$. Applying the Fourier-Budan Theorem (see Theorem \ref{Fourier-Budan Theorem}), we check that the sequence of derivatives at $V=-1$ is $\{4(m+1-2mz), -4(m+1)-9m\ga z, -18+18m\ga z, 54\}$,  and the sequence at $V=0$ is $\{-8m z, 5-4m+9m\ga z,36+18m\ga z, 54\}$. Hence, the difference of the variations of signs at $V=-1$ and $V=0$ is 1. It follows that $p(V)$ only has one zero in the interval $V\in[-1,0]$. Given that $p(0)<0$ and $p(-1)>0$, we will prove that $V_4\le -\frac{2}{5}$ and $p(-\frac{2}{5})>0$. 
	Recalling \eqref{P4}, we have $\frac{dV_4(z)}{dz} = \frac{-2m\ga}{(m+1)\ga+1-m}<0$. Together with Lemma \ref{lower bound of zg}, it is enough to show that $V_4(\frac{1}{10})\le -\frac{2}{5}$, which follows directly from
	\begin{align*}
	V_4(\frac{1}{10})+\frac{2}{5}=\frac{-8+(m+2)\ga-2m}{5[(m+1)\ga+1-m]} \leq 0.
	\end{align*}
	Also, by using $z\leq z_M(\ga)<\frac{1}{\ga+2+2\sqrt{5}}$ for any $\ga\in[\ga_g,3]$, we obtain
	\begin{align*}
		p(-\frac{2}{5}) = \frac{38}{125}+\frac{m}{25}[40-(200+54\ga)z]>\frac{m(80\sqrt 5-120-14\ga)}{25(\ga+2+2\sqrt{5})}>0, 
	\end{align*}
	for any $\ga\in[\ga_g,3]$. Hence, $p(V_4(z))>0$ for any $\ga\in[\ga_g,3]$ and $z\in[z_g,z_M]$. This completes the proof.
\end{proof}

\begin{lemma}\label{barrier P5}
	Consider $P_*=P_8$. For any $\ga\in[\ga_g,3]$ and $z\in[z_g,z_M]$, the solution trajectory, as guaranteed by Lemma \ref{intersect G=0}, approaches the sonic line $C=-V-1$ at $(-C_s-1,C_s)$ with $C_s<C_5$.
\end{lemma}

\begin{proof}
	We will show that the $V(C)=-\frac{2}{3}C^2$ is a lower barrier for the solution curve from $P_8$ in the interval $C\in[C_5,0)$. By employing the barrier argument, it is enough to show that
\begin{align*}
	\frac{d}{dC} (V + \tfrac{2}{3}C^2) \Big | _{( -\frac{2}{3}C^2 ,C)}= \frac{G(-\frac{2}{3}C^2,C;\ga,z)}{F(-\frac{2}{3}C^2,C;\ga,z)}+\frac{4}{3}C <0 
\end{align*}
for any  $\ga\in[\ga_g,3]$, $z\in[z_g,z_M]$ and $C\in[C_5,0)$. 
Since $F(-\frac{2}{3}C^2,C;\ga,z)>0$ and $1-\frac{2}{3}C^2>0$ for any $C\in[C_5,0)$, $z\in[z_g,z_M]$ and $\ga\in[\ga_g,3]$, this is equivalent to showing that
\begin{align}
	\Big[G(-\frac{2}{3}C^2,C;\ga,z)+\frac{4}{3}CF(-\frac{2}{3}C^2,C;\ga,z) \Big]\Big(1-\frac{2}{3}C^2\Big)<0. 
\end{align}
A direct computation shows that
\begin{align*}
	\Big[G(-\frac{2}{3}C^2,C;\ga,z)+\frac{4}{3}CF(-\frac{2}{3}C^2,C;\ga,z) \Big]\Big(1-\frac{2}{3}C^2\Big) = \frac{C^2}{3} \Barrier(\frac{2}{3}C^2),
\end{align*}
where
\begin{align*}
	\Barrier(\xi) =&\, 2[m(\ga-1)+1]\xi^3+[7m-9-4m\ga+2m(\ga-2)\ga z]\xi^2-[5m-9-2m\ga+2m(\ga-3)\ga z]\xi\\
	&-2+2m(3-\ga)z.
\end{align*}
Our goal is to show that $\Barrier(\xi)<0$ for any  $\ga\in[\ga_g,3]$, $z\in[z_g,z_M]$ and $\xi\in(0,\frac{2}{3}C_5^2]$. We compute
\begin{align*}
	\Barrier'(\xi) &= 6[m(\ga-1)+1]\xi^2+2[7m-9-4m\ga+2m(\ga-2)\ga z]\xi+9-5m+2m\ga+2m(3-\ga)\ga z \\
	&>4m(\ga-1)(1-\xi)^2+m+3-2m\ga\\
	&>\frac{16}{9}m(\ga-1)+m+3-2m\ga = -\frac{2}{9}m\ga-\frac{7m}{9}+3>0,
\end{align*}
for any $\ga\in[\ga_g,3]$ and $z\in[z_g,z_M]$, where we have used $0<\xi\leq\frac23 C_5^2<\frac{8}{27} (<\frac13)$ from Lemma \ref{C4<2/3} in the second inequality. Hence, it is sufficient to check $\Barrier(\frac{8}{27})<0$ by Lemma \ref{C4<2/3}. By direct computation,
\begin{align*}
	\frac{1}{2}\Barrier(\frac{8}{27}) &= m[(3-\frac{209\ga}{729}-\frac{152\ga^2}{729})z+\frac{2888\ga-9044}{19683}]-\frac{703}{19683}\\
	&<m[(3-\frac{209\ga}{729}-\frac{152\ga^2}{729})z_M+\frac{2888\ga-9044}{19683}]-\frac{703}{19683}<0,
\end{align*}
where we have used $z\leq z_M$ and $3-\frac{209\ga}{729}-\frac{152\ga^2}{729}>0$ for any $\ga\in[\ga_g,3]$ in the first inequality, and the negativity of the identity is shown in Proposition~\ref{prop:poly} \eqref{Poly L 3.4}.
\end{proof}

\subsection{Continuation of trajectories from $P_6$}
When considering the solutions from $P_8$, we have implicitly employed the fact that the point $P_9$ remains away from the $C$-axis. When $P_*=P_6$, for $\gamma$ close to 1, $z_{{std}}$  approaches $z_g$ (which in turn approaches zero), and so $P_9$ is in close proximity to the $C$-axis. In this scenario, we do not anticipate the existence of a suitable lower barrier function in the third quadrant. Hence, it becomes crucial to distinguish between larger and smaller values of $z$ by employing distinct methods, and to understand the trajectory of the solution in the fourth quadrant before it possibly passes  through to the third quadrant. Referring to Theorem \ref{collapsingthm}, our analysis is confined to the range $\ga\in(1,2)$.

To quantify this distinction between smaller and larger values of $z$, we introduce a function $z_0(\ga)$, which provides an effective separation between the regime in which we may apply a barrier argument in the third quadrant (the case $z\geq z_0$), and the regime in which we apply a direct integration estimation procedure ($z<z_0$). The barrier function that we wish to employ for $z\geq z_0(\ga)$ is $V=V(C)=(k(\ga,z)+\epsilon)C^2$, where
$k(\ga,z) = V_8(\ga,z)(C_8(\ga,z))^{-2}$
and $\epsilon>0$ is sufficiently small and guarantees that $V(\Cg)>V_8$. In order to close the barrier argument for $C$ close to $C_9$, however, we find that we cannot take $z$ arbitrarily close to $z_g$ for the reasons outlined above, and so we design $z_0(\ga)$ such that the barrier argument closes and we have a tractable form for analysis. This motivates the definition
\begin{align}\label{z0}
z_0(\ga) = \frac{22-5\ga}{125},
\end{align}
compare Step 1 of the proof of Lemma~\ref{lower barrier when m=1,2} below. We then find (see Lemma \ref{z0=zg}) that there exists a unique $\gau\in(1,2)$ (in fact, $\gau<\frac85$) such that $z_0(\ga_u)=z_g(\ga_u)$. In particular, $z_0>z_g$ when $\ga\in(1,\gau)$, and $z_0< z_g$ when $\ga\in(\gau,2)$.
For notational convenience, we introduce
\begin{align}
	z_s(\ga) = \begin{cases}&z_0(\ga),\quad\text{when }\ga\in(1,\gau),\\
	&z_g(\ga),\quad\text{when }\ga\in[\gau,2). \end{cases}	
\end{align}
We then argue via  the barrier argument and a perturbation argument (to allow $\epsilon>0$ in the barrier function) for $\ga\in(1,2)$ and $z\in[z_s(\ga),z_M(\ga)]$, alongside a direct integration estimation argument for $\ga\in(1,\gau)$ and $z\in(z_g(\ga),z_0(\ga)).$ We initiate the analysis with the larger $z$.

\subsubsection{Continuation of trajectories from $P_6$ case I: $z\in[z_s,z_M]$}
Our objective is to show that the solution trajectory intersects $C=-(1+V)$ at $(-C_s-1,C_s)$  with $C_s<\Cg$. Consequently, to employ a barrier argument, the chosen lower barriers should satisfy $V(\Cg)>V_8$. As discussed above, a suitable candidate is given by $V=V(C)=(k(\ga,z)+\epsilon)C^2$, where
\begin{align}
k(\ga,z) = \frac{V_8(\ga,z)}{C_8^2(\ga,z)}\label{k(ga,z)}<0,
\end{align}
and $\epsilon>0$ is a sufficiently small number. The following lemma provides a quantitative property of $k(\ga,z)$, which will be instrumental to the later analysis.
\begin{lemma}\label{upper bound of k(ga,z)}
For any $\ga\in(1,2)$ and $z\in [z_s(\ga),z_M(\ga)]$,
 \begin{align*}
 	k(\ga,z)<-\frac{3}{10}.
 \end{align*}
\end{lemma}

\begin{proof}
We first notice that, by using \eqref{ineq: V'6>0, V'8<0},
\begin{align}\label{k'(z)<0}
	\frac{\partial k(\ga,z)}{\partial z} = \frac{(1-V_8)V_8'(z)}{(1+V_8)^3}<0. 
\end{align}
Recalling also $z_s(\ga)\geq z_0(\ga)> z_0(2)$ for $\ga\in(1,2)$, it is sufficient to show that $k(\ga,z_0(2))<-\frac{3}{10}$. We directly compute and verify 
\begin{align*}
	(1+&V_8(z_0(2)))^2 [k(\ga,z_0(2))+\frac{3}{10}]=V_8(z_0(2))+\frac{3}{10}(1+V_8(z_0(2)))^2\\
	&=\frac{18 (-72 \ga^2 - 2980 \ga +335)}{5[-432\ga^2-13272\ga+155647+(36\ga+1553)\sqrt{144\ga^2-3576\ga+10201}]}<0
\end{align*}
for any $\ga\in(1,2)$.
\end{proof}

The following is the main result in this subsection. It establishes the existence of a barrier function that guarantees that the maximal extension coordinate $C_s<C_9$ when $\ga\in(1,2)$ and $z\in [z_s(\ga),z_M(\ga)]$.
\begin{lemma}\label{lower barrier when m=1,2}
Consider $P_*=P_6$ and $m=1,2$. For any $\ga\in(1,2)$ and $z\in [z_s(\ga),z_M(\ga)]$, there exists a sufficiently small $\epsilon>0$ such that for $\hat k\in [k(\ga,z),k(\ga,z)+\epsilon]$, the function $V(C)=\hat k C^2$ is a lower barrier for solution trajectories,  
as provided by Lemma \ref{pass collapse}, for $C\in[C_9,0)$.
\end{lemma}

\begin{proof}
By employing the barrier argument and using $F(\hat kC^2,C;\ga,z)>0$ by Lemma \ref{L:FG=0} and $1+\hat kC^2>0$ for any $C\in[C_9,0)$, $z\in[z_s,z_M]$ and $\ga\in(1,2)$, it is sufficient to show that
\begin{align}\label{barrier inequality}
	\Big[G(\hat kC^2,C;\ga,z)-2\hat kCF(\hat kC^2,C;\ga,z) \Big]\Big(1+\hat kC^2\Big)<0 .
\end{align}
By direct computation from \eqref{G(V,C)}--\eqref{F(V,C)}, we find
\begin{align}
	\Big[G(\hat kC^2,C;\ga,z)-2\hat kCF(\hat kC^2,C;\ga,z) \Big]\Big(1+\hat kC^2\Big)= C^2 \Barrier (\hat k,\hat kC^2 ;m),
\end{align}
where 
\begin{align}\label{eq: k}
\Barrier (\hat k, \xi;m) :=&[m(\ga-1)+1]\hat k\xi^3+[3+2m(\ga-1)+m(2-\ga)\ga z]\hat k\xi^2+(m-1)\xi^2\notag\\ &+[3+m(\ga-1)+m(3-\ga)\ga z]\hat k\xi+(m-1)\xi+2mz+(1+m\ga z)\hat k.
\end{align}
It is clear that $\xi=\hat k C^2$ takes values in the interval $[V_8,0)$ for $C\in[C_9,0)$. Our goal is to show $\Barrier (\hat k, \xi;m)<0$, and the strategy is the following: 
	\begin{enumerate}
		\item For $m=1,2$, we will first show that $ \Barrier (\hat k, V_8;m)<0$ in the interval $\hat k\in(k(\ga,z),k(\ga,z)+\epsilon]$ for sufficiently small $\epsilon>0$.
		\item For $m=1,2$, we will prove that $\frac{\partial \Barrier (\hat k, \xi;m)}{\partial \xi}<0$ for $\xi\in[V_8,0)$.
		\end{enumerate}

$\underline{\text{Step 1}}:$ Since it is straightforward to see $\Barrier (k(\ga,z), V_8;m)=0$ from the vanishing of $G$ and $F$ at $(V_8,C_8)$, the existence of such small and positive $\epsilon$ can be obtained by establishing the negativity of the following:
	\begin{align*}
	\frac{\partial \Barrier(\hat k,V_8;m)}{\partial \hat k}|_{\hat k = k(\ga,z)} 	=\frac{m}{2}\alpha_1(\ga,z)+\alpha_2(\ga,z),
	\end{align*}
	where
	\begin{align*}
		\alpha_1(\ga,z) &= (\ga-2)^4z^3+(2-\ga)(2\ga^2+\ga-2)z^2+(\ga^2-4\ga-4)z+1+[(\ga-2)^3z^2-(\ga^2+\ga+2)z+1]w(z),\\
		\alpha_2(\ga,z) &= (\ga-2)^3z^3+(2-\ga)(\ga+4)z^2-2z+[(\ga-2)^2z^2-2z]w(z).
	\end{align*}
	Since for any $\ga\in(1,2)$, $z<z_M<\frac{1}{5}$, it is clear that $\alpha_2<0$. When additionally $\alpha_1(\ga,z)\leq 0$, then clearly $\frac{\partial \Barrier(\hat k,V_8;m)}{\partial \hat k}|_{\hat k = k(\ga,z)}$ is negative for both $m=1,2$. If, on the other hand, $\alpha_1(\ga,z)>0$, then $\frac{\partial \Barrier(\hat k,V_8;1)}{\partial \hat k}|_{\hat k = k(\ga,z)}<\frac{\partial \Barrier(\hat k,V_8;2)}{\partial \hat k}|_{\hat k = k(\ga,z)}$. Hence, it is enough to check the following inequality holds for any $\ga\in(1,2)$ and $z\in[z_s,z_M]$:
	\begin{align}\label{Ineq: H'(m=2)}
	 q(z)w(z)+p(z):=\frac{\partial \Barrier(\hat k,V_8;2)}{\partial \hat k}|_{\hat k = k(\ga,z)} <0,
	\end{align}
	 where
	\begin{align*}
		p(z) &=(\ga-2)^3(\ga-1)z^3+2(2-\ga)(\ga^2+\ga+1)z^2+(\ga^2-4\ga-6)z+1,\\
		q(z) &= (\ga-2)^2(\ga-1)z^2-(\ga^2+\ga+4)z+1.
	\end{align*}
	By Lemma \ref{p(z)<0}, $p(z)<0$. If also $q(z)\leq 0$, the inequality \eqref{Ineq: H'(m=2)} is trivially satisfied. If $q(z)>0$, by using $w(z)<\frac12$ as shown in Lemma \ref{w(zs)<1/2}, it is enough to check $r(z):=q(z)+2p(z)<0$, where $r(z)$ is considered as 
	a polynomial in $z$ for each fixed $\ga$. 
	We compute
	\begin{align*}
		r'(z) = 6(\ga-2)^3(\ga-1)z^2-2(\ga^2-4)(3\ga+1)z+\ga^2-9\ga-16
		< \ga^2 -4\ga-16<0, 
	\end{align*}
	where we used $z\le z_M<\frac{1}{5}$. 
	Hence, $r(z)<r(z_s)$ for $z\in (z_s, z_M]$. For $\ga\in(1,\gau)$, by Proposition~\ref{prop:poly} \eqref{r(z0)<0}, we have $r(z_0)<0$. If $\ga\in[\gau,2)$, $z_s=z_g$. By using $r'(z)<0$ and Lemma \ref{z0=zg}, it is enough to verify $r(z_g)<r(z_0(\frac{159}{100}))<0$ for $m=1$ and $r(z_0(\frac{31}{20}))<0$ for $m=2$. The proofs of these two inequalities are presented in Proposition~\ref{prop:poly} \eqref{r(z0(1.59))<0} and Proposition~\ref{prop:poly} \eqref{r(z0(1.55))<0}, respectively. Hence, \eqref{Ineq: H'(m=2)} holds. This completes the proof of Step 1, and the existence of $\epsilon>0$ follows from a continuity argument.

$\underline{\text{Step 2}}:$ Let $\epsilon>0$ be given as in Step 1. Given that $ \Barrier (\hat k, V_8;m)<0$, we shall show $\Barrier' (\hat k, \xi;m) := \frac{\partial \Barrier(\hat k,\xi;m)}{\partial \xi}<0$. 
We treat $m=1$ and $m=2$ separately. When $m=1$, direct computation shows
\begin{align*}
	\Barrier' (\hat k,\xi;1)&= 3\ga \hat k\xi^2+2(1+2\ga +\ga z(2-\ga))\hat k \xi + (2+\ga+\ga z(3-\ga))\hat k\\
	&\leq \Big[2(- \xi^2+\xi+ 1) + \ga z(\ga-1)\Big]\hat k<0,
\end{align*} 
where we  have used  $\hat k<0$, $1<\ga< 2$ and $-1<V_8\leq \xi<0$, and $\frac{1-\sqrt{5}}{2}<\frac{-\sqrt{2}}{\sqrt{2}+\sqrt{\gamma}}\leq V_8\leq \xi<0$. 

In the case   $m=2$, we obtain
\begin{align*}
	\Barrier' (\hat k,\xi;2) 
	&=\Big[(4\ga-1)(1+\xi)^2+2(1-\ga)+2(\ga-1)\xi^2+2\ga z[2(2-\ga)\xi+3-\ga] \Big]\hat k+2\xi+1.
\end{align*}
We first claim that $\Barrier'(\hat k,\xi;2)$ is increasing in $\hat k$. Since it 
is linear in $\hat k$, we only need to check the positivity of the coefficient of $\hat k$. 
For any $\ga\in(1,2)$, $z\in[z_s,z_M]$ and $\xi\in[V_8,0)$, the coefficient satisfies 
\begin{align*}
	&(4\ga-1)(1+\xi)^2+2(1-\ga)+2(\ga-1)\xi^2+2\ga z[2(2-\ga)\xi+3-\ga \\
	&>(4\ga-1)C_8^2(z)+2(1-\ga)+2\ga z[2(2-\ga)V_8(z)+3-\ga]=:p(z). 
\end{align*}
Hence, it is enough to show $p(z)>0$. To this end, we will show that $p$ is concave on $(0,z_M]$, $p(0)>0$ and $p(z_M)>0$. 
By using $C_8'(z)=V_8'(z)<0$,  $V''_8<0$ and $C_8 C_8''(z)+(C'_8)^2<0$ by \cite[Lemma 6.1]{JLS23},  we deduce that 
\begin{align*}
	p''(z)=2(4\ga-1)[C_8 C_8''(z)+(C'_8)^2]+4\ga z(2-\ga)V_8''(z)+8\ga (2-\ga)V_8'(z)<0, 
\end{align*}
for any $z\in(0,z_M]$. 
Next, from \eqref{Id: VCzM},
\begin{align*}
	p(z_M) = (4\ga-1)C_8^2(z_M)+2(1-\ga)+2\ga z_M[2(2-\ga)V_8(z_M)+3-\ga] = \frac{4 + 4 \sqrt{2\ga}  - 5 \ga}{(\sqrt\ga+\sqrt 2)^2}>0
\end{align*}
for any $\ga\in(1,2)$. Moreover, by using $C_8(0)=1$ and $V_8(0)=0$, we obtain $p(0) = 2\ga+1>0$.
 Therefore, $p(z)>0$ for any $z\in (0,z_M]$ which implies that $\Barrier'(\hat k,\xi;2)$ is increasing in $\hat k$. Together with Lemma \ref{upper bound of k(ga,z)}, Lemma \ref{lower bound of zg} and \eqref{k'(z)<0}, we obtain an upper bound of $\Barrier'(\hat k,\xi;2)$: 
\begin{align*}
	\Barrier'(\hat k,\xi;2) < \frac{-9(2\ga-1)}{10} \xi^2+[2-\frac{3(4\ga-1+2(2-\ga)\ga z)}{5}]\xi-\frac{3(2\ga+1+2(3-\ga)\ga z)}{10}+1: = b(\xi).
\end{align*}
Now, in order to complete Step 2, it suffices to show $b(\xi)<0$ on $[V_8,0)$. We check the discriminant
\begin{align*}
\Delta
&=\frac{2}{25}\Big[18\ga^2(\ga-2)^2z^2+(-18\ga^2+33\ga-75)\ga z+18\ga^2-66\ga+53\Big]=:\frac{2}{25}q(z).
\end{align*}
We claim that $q(z)<0$. 
We observe that for $\ga\in (1,2)$, $z_0(2)< z_s < z_M <1$. 
As $q(z)$ is clearly convex in $z$, $q(z_0(2))<0$ is shown in Proposition~\ref{prop:poly} \eqref{L 3.7-1}, and $q(1)<0$ is shown in Proposition~\ref{prop:poly} \eqref{L 3.7-2}, the discriminant is negative, and hence $b(\xi)<0$. Consequently,
\begin{align*}
	\Barrier'(\hat k,\xi;2)<b(\xi)<0
\end{align*}
for any $\ga\in(1,2)$, $z\in[z_s,z_M]$ and $\xi\in[V_8,0)$. This concludes the proof.
\end{proof}

\subsubsection{Continuation of trajectories from $P_6$ case II: $z\in(z_g,z_s)$}
In this subsection, we consider $\ga\in(1,\gau)$ and $z\in(z_g,z_0)$. As discussed before,  due to the proximity of $P_9$ to the $V$-axis in this range of parameters, we must be 
careful in establishing suitable barriers for the flow after passing through the origin. 
We therefore 
seek a barrier in the fourth quadrant, 
rather than the third quadrant, that ensures that the flow passes below $C_9$ before approaching the sonic line. In order to achieve this, we employ the barrier function $B(C)=-C(1+C)$ in Lemma~\ref{m2 small ga and z}. However, to use this as a barrier function close to the origin, we require more precise control on the slope at which the solution trajectory passes through the origin. Our initial focus is therefore on estimating a quantitative lower bound on $\lim_{V\to 0^{-}}C'(V)$, which provides insight into the behavior of $V(C)$ around $P_0$. This lower bound will be demonstrated via an approximate integration argument of the ODE \eqref{ODE} for the solution trajectory in the second quadrant in Lemma~\ref{upper bound of sol}.

 We begin by establishing two useful quantitative properties.  
\begin{lemma}\label{phi(1)<phi(2)}
For any $\ga\in(1,\frac{8}{5})$ and $z\in(0,z_0(\ga))$,
\begin{align*}
	\Big(\frac{1+V_8+\ga z}{1+V_6+\ga z}\Big)^2<\Big(\frac{1+V_8+2\ga z}{1+V_6+2\ga z}\Big)^3.
\end{align*}
\end{lemma}

\begin{proof}
By direct computation from~\eqref{P6}--\eqref{P8},
\begin{align*}
	&(1+V_8+\ga z)^2(1+V_6+2\ga z)^3-(1+V_6+\ga z)^2(1+V_8+2\ga z)^3\\
	 &= [(5\ga^2-4)z^2+(8-\ga)z-4]w(z)\ga^2 z^2<\Big(\frac{5\ga^2-4}{25}+\frac{8-\ga}{5}-4\Big)w(z)\ga^2 z^2<0,
\end{align*}
where we have used that $5\ga^2-4>0$, $8-\ga>0$ and $z_0<z_M<\frac{1}{5}$  for any $\ga\in(1,\frac{8}{5})$ in the first inequality.
\end{proof}

\begin{lemma}\label{upper bound CV8}
Consider $P_*=P_6$ and $m=1,2$. For any $\ga\in(1,\gau)$ and $z\in(z_g(\ga),z_0(\ga))$, the solution curve $C=C(V)$ given by Theorem \ref{collapsingthm} satisfies
\begin{align}\label{upper bound of C(V8)}
	C(V_8)<\frac{2C_6}{3}, 
\end{align}
in the second quadrant ($\{(V,C)\mid V<0, C>0\}$).
\end{lemma}

\begin{proof}
	Recalling Lemma \ref{intersect G=0}, we notice that for any $V\in(V_6,V_8)$,
	\begin{align*}
		\cfrac{C^2f_1-f_2}{C^2g_1-g_2}-\cfrac{f_2}{g_2} =-\frac{m[(m + 1)(\gamma- 1) + 2]}{2}\cfrac{(V-V_4)(V-V_6)(V-V_8)C^2}{(C^2g_1-g_2)g_2}<0,
	\end{align*}
where we used $V_4<V_6<V_8$ for any $\ga\in(1,\gau)$ and $z\in(z_g(\ga),z_0(\ga))$, and $g_2(V)<0$ and $C^2g_1-g_2>0$ along the solution curve for $V\in(V_6,0)$. Thus, for any $V\in(V_6,V_8)$,
	\begin{align*}
		\ln C(V_8) < \int_{V_6}^{V_8} \frac{f_2(V)}{g_2(V)} dV+\ln C_6 =:\ln\Phi(\ga,z;m).
	\end{align*}
	In order to show \eqref{upper bound of C(V8)} holds, it is therefore sufficient to show that $\Phi(\ga,z;m)<\frac{2C_6}{3}$. By direct computation from~\eqref{g1g2f1f2},
	\begin{align*}
		\Phi(\ga,z;m) = C_6\frac{V_8}{V_6} \Big(\frac{1+V_8}{1+V_6} \Big)^{\frac{1-\ga}{2}}\Big(\frac{1+m\ga z+V_8}{1+m\ga z+V_6} \Big)^{\frac{(m+1)(\ga-1)}{2}}.
	\end{align*}
	From Lemma \ref{phi(1)<phi(2)}, $\Phi(\ga,z;1)<\Phi(\ga,z;2)$ for any $\ga\in(1,\frac{8}{5})$ and $z\in (0,z_0)$. 
	It is therefore enough to check $\Phi(\ga,z;2)<\frac{2C_6}{3}$ for $\ga\in(1,\frac{8}{5})$  (as $\gau<\frac85$) and $z\in(0,z_0)$. We observe that
	\begin{align}
		\Phi(\ga,z;2) = C_6\Big(\frac{V_8}{V_6} \Big)^{\frac{5-3\ga}{2}}\Big(\frac{(1+V_8)^2V_8^3}{(1+V_6)^2V_6^3} \Big)^{\frac{\ga-1}{2}}\Big(\frac{1+V_6}{1+V_8} \frac{1+2\ga z+V_8}{1+2\ga z+V_6} \Big)^{\frac{3(\ga-1)}{2}}.
	\end{align}
	The validity of the following claims:
	\begin{align}
	\frac{(1+V_6)(1+V_8+2\ga z)}{(1+V_8)(1+V_6+2\ga z)}&<\frac{2}{3}, \  \ \ 
	\frac{(1+V_8)^2V_8^3}{(1+V_6)^2V_6^3}<1, \  \ \ 
	\frac{V_8}{V_6} <\frac{4}{9}, \label{3.9-3}
	\end{align}
	implies
	\begin{align*}
		\Phi(\ga,z;2) < C_6 (\frac{2}{3})^{\frac{7-3\ga}{2}}<\frac{2}{3}C_6
	\end{align*}
	for any $\ga\in(1,\frac{8}{5})$. Hence, our goal turns to verify the three inequalities in 
	\eqref{3.9-3}.
	To show the first one, 
	we check that for any $\ga\in(1,\frac{8}{5})$ and $z\in(0,z_0)$,
		\begin{align*}
		&3(1+V_6)(1+V_8+2\ga z)-2(1+V_8)(1+V_6+2\ga z) = [2+(\ga-2)z-5w(z)]\ga z\\ 
		&= \frac{\ga z[-24(\ga-2)^2z^2+(92+54\ga)z-21]}{2+(\ga-2)z+5w(z)}<\frac{\ga z[-24(\ga-2)^2z_0^2+(92+54\ga)z_0-21]}{2+(\ga-2)z+5w(z)}<0,
	\end{align*}
	where the negativity of the numerator is shown in Proposition~\ref{prop:poly} \eqref{L 3.13-1}. 	For the second inequality in \eqref{3.9-3}, 
	we compute
	\begin{align*}
	(1+V_8)^2V_8^3-(1+V_6)^2V_6^3 &= [(\ga-2)^4z^2-2(\ga-2)^2(\ga+1)z+\ga^2]w(z)z^2\\
	&>[\ga^2-2(\ga+1)/5]w(z)z^2>0,
	\end{align*}
		where we used that $z<z_M<\frac{1}{5}$ and $\ga\in(1,\frac{8}{5})$ in the   inequality. Finally, we check the last inequality in \eqref{3.9-3}. 
	We compute
	\begin{align*}
		8V_6-18V_8 
		= 5+5(2-\ga)z-13w(z) 
		 = \frac{-8(18(\ga-2)^2z^2-(36\ga+97)z+18)}{5+5(2-\ga)z+13w(z)}=:\frac{-8p(z)}{5+5(2-\ga)z+13w(z)}.
	\end{align*}
	It is clear that $p(z)$ achieves the global minimum at $z=\frac{36\ga+97}{36(\ga-2)^2}>1>z_0$. Hence, it is enough to check the positivity of $p(z_0)$, 	which is given by Proposition~\ref{prop:poly} \eqref{L 3.13-2}. Therefore, $8V_6-18V_8<0$ for each $\ga\in(1,\gau)$ and $z\in(z_g,z_0)$.
	 This completes the proof.
\end{proof}

Next, we engage in an approximate integration procedure for our solutions in the second quadrant to obtain better control of the behavior around $P_0$. In particular, we obtain an upper bound for the solution trajectory in the second quadrant, which implies a lower bound for the trajectory leaving the origin into the fourth quadrant.
\begin{lemma}\label{upper bound of sol}
Consider $P_*=P_6$ and $m=1,2$. For any $\ga\in (1,\gau)$ and $z\in (z_g(\ga),z_0(\ga))$, $B(V) =-V$ is an upper bound in the 2nd quadrant of the solution trajectories $C=C(V)$ given by Theorem \ref{collapsingthm} for any $V\in[V_8,0)$.
\end{lemma}
\begin{proof}
We easily see that $C(V)<-V$ for $V\in[V_6,V_8]$ from the simple observation that the trajectory is monotone and 
\begin{align}\label{C6<-V8}
V_8 + C_6 = C_8+C_6-1 = (\ga-2)z<0,\quad\text{ so that } \quad C(V_8)<C_6<-V_8.
\end{align}
Thus, by continuity of the trajectory, there exists a maximal $V_b\in(V_8,0]$ such that $C(V)<-V$ on the interval $[V_8,V_b)$. To conclude the proof of the Lemma, it is sufficient to show that $V_b=0$. 

Suppose for a contradiction that $V_b<0$. Then, by continuity, we have 
\begin{align}\label{assumption}
	C(V_b)=-V_b.
\end{align}
We will show \eqref{assumption} holds only when $V_b=0$. 
We first observe from~\eqref{G(V,C)}--\eqref{F(V,C)} that
\begin{align*}
	\frac{\partial}{\partial C}\Big(\frac1C\frac{F(V,C)}{G(V,C)}\Big) = \frac{2C(g_1f_2-g_2f_1)}{(C^2g_1(V)-g_2(V))^2} =  \frac{m[(m + 1)(\gamma- 1) + 2]}{2}\frac{2C (V-V_4)(V-V_6)(V-V_8)}{(C^2g_1(V)-g_2(V))^2} >0
\end{align*}
for any $V\in[V_8,0)$, $C>0$, where we used $V_4<V_6<V_8$ for any $\ga\in(1,\gau)$ and $z\in(z_g(\ga),z_0(\ga))$.
Therefore,  as $0<C( V)<- V$ on $[V_8,V_b)$, we have
\begin{align*}
    \ln{C(V_b)} = \ln{C(V_8)} +\int_{V_8}^{V_b} \cfrac{C^2f_1(V )-f_2(V )}{C^2g_1(V )-g_2(V )} \,d{V} <  \ln{C(V_8)} +\int_{V_8}^{V_b} \cfrac{V^2f_1(V )-f_2(V )}{V^2g_1(V )-g_2(V )} \,dV=: \ln U(V_b).
\end{align*}
Recalling the identities \eqref{g1g2f1f2} and \eqref{a_1234&z}, we solve for $\ln U(V)$ using the partial fraction decomposition. A lengthy calculation shows that
\begin{align*}
     &\cfrac{V^2f_1(V)-f_2(V)}{V^2g_1(V)-g_2(V)}
         =\frac{1}{V}+\frac{b_1}{1+V}+\frac{b_2}{V-r_1}+\frac{b_3}{V-r_2},
\end{align*}
where
\begin{align}
\Delta &=m^2(\ga-2)^2z^2+4m[(m+1)\ga-2]z+4(m+1) > m^2(\sqrt 3+(\ga-1)z)^2,\label{Delta}\\
 r_1 &= \cfrac{-[m(2-\ga)z-2]+\sqrt{\Delta}}{2m}>\frac{2+\sqrt 3 m+m(2\ga-3)z}{2m}>0,\label{r1}\\
    r_2 &= \cfrac{-[m(2-\ga)z-2]-\sqrt{\Delta}}{2m},\label{r2}\\
    b_1 &= -\frac{mz}{m+1-2mz}<0,\label{B}\\
b_2&= \frac{1-\ga+(\ga-1)\ga z-4z+\frac{2(1+m\ga z)}{m}b_1+(-\ga-1-2b_1)r_1}{2(r_1-r_2)}<0,\label{Q}\\
b_3&=\frac{1-\ga+(\ga-1)\ga z-4z+\frac{2(1+m\ga z)}{m}b_1+(-\ga-1-2b_1)r_2}{2(r_2-r_1)},\label{W}
\end{align}
where the inequality \eqref{Delta} is obtained by directly verifying that 
\begin{align*}
	\Delta - m^2(\sqrt 3+(\ga-1)z)^2= - m^2(2\ga-3)z^2 +2m(4-\sqrt 3 m)(\ga-1)z + 4(m+1)-3m^2 +4m(m-1)\ga z>0
\end{align*}
holds for both $m=1,2$, $z\in(z_g,z_0)$ and $\ga\in(1,\gau)$. Inequalities \eqref{r1} and \eqref{Q} then follow as well. 
It is important to note that the validity of $r_2<V_8$ is necessary to ensure that $\ln U(V)$ is well-defined in the interval $V\in(V_8,0)$. By using $w'(z)<0$ (see \eqref{ineq: V'6>0, V'8<0}), we check that
\begin{align*}
V_8 - r_2 =&\,\frac{\sqrt\Delta-[2+m(1-w(z))]}{2m}>\frac{\sqrt\Delta-(2+\frac{3m}{5})}{2m} \\
&=\frac{m^2(\ga-2)^2z^2+4m[(m+1)\ga-2]z+\frac{8m}{5}-\frac{9m^2}{25}}{2m[\sqrt\Delta+(2+\frac{3m}{5})]}
>\frac{40-9m}{50[\sqrt\Delta+(2+\frac{3m}{5})]}>0,
\end{align*}
where we have used Lemma \ref{w(zs)>2/5} in the first inequality. Therefore, we have
\begin{align}\label{ineq:U(V)<-V}
    U(V_b) = -V_b\frac{C(V_8)}{-V_8}\Big(\frac{1+V_b}{1+V_8}\Big)^{b_1}\Big(\frac{r_1-V_b}{r_1-V_8}\Big)^{b_2}\Big(\frac{V_b-r_2}{V_8-r_2}\Big)^{b_3},
\end{align}
which is well-defined as $V_b\in(V_8,0)$. 

 From Lemma \ref{upper bound CV8} and \eqref{C6<-V8}, 
 $\frac{C(V_8)}{-V_8}<-\frac{2C_6}{3V_8}<\frac{2}{3}$. Additionally, since $b_1<0$, it follows that $(\frac{1+V_b}{1+V_8})^{b_1}<1$. It is then enough to show that, for any $V\in(V_8,0)$,
\begin{align*}
	\phi(V):=\frac{2(r_1-V)^{b_2}(V-r_2)^{b_3}}{3(r_1-V_8)^{b_2}(V_8-r_2)^{b_3}}<1,
\end{align*}
as this establishes $U(V_b)<-V_b$, a contradiction.
Considering the intricate nature of the function $\phi(V)$, which depends on $V$, $\gamma$, and $z$, we aim to establish segmented bounds for $\phi(V)$. We first decompose $\phi(V)$ as
\begin{align*}
	\phi(V) = \frac{2}{3} (\frac{r_1-V}{r_1-V_8})^{b_2+b_3} (\frac{r_1-V_8}{r_1-V}\frac{V-r_2}{V_8-r_2})^{b_3}
\end{align*}
and reduce the problem to showing
\begin{align}
	(\frac{r_1-V_8}{r_1-V}\frac{V-r_2}{V_8-r_2})^{b_3}&<1,\quad 
	(\frac{r_1-V}{r_1-V_8})^{b_2+b_3} <\frac{3}{2}\label{phi2}.
\end{align}
To prove the first inequality of \eqref{phi2}, we first show that $b_3<0$. Denote $p(z) = 2(r_2-r_1)b_3$. We compute
\begin{align*}
p'(z) = (\ga-1)\ga-4+2\ga b_1+(\frac{2(1+m\ga z)}{m}-2r_2)b_1'(z) - (\ga+1+2b_1)r_2'(z),
\end{align*}
where
\begin{align*}
	b_1'(z) &= -\frac{m(m+1)}{(m+1-2mz)^2}<-\frac{m}{m+1}<0,\\
	r_2'(z) 
	&=-\frac{2-\ga}{2}-\frac{m}{2}\sqrt{\frac{(\ga-2)^4z^2+\frac{4}{m}[(m+1)\ga-2](\ga-2)^2z+\frac{4}{m^2}[(m+1)\ga-2]^2}{m^2(\ga-2)^2z^2+4m[(m+1)\ga-2]z+4(m+1)}}\\&>-\frac{2-\ga}{2}-\frac{m}{2} = \frac{\ga-2-m}{2},
\end{align*}
for $m=1,2$, $z\in(z_g,z_0)$ and $\ga\in(1,\gau)$. It is clear that $-\ga-1-2b_1<0$ and $\frac{2(1+m\ga z)}{m}-2r_2>0$ as $r_2<V_8<0$. Therefore,
\begin{align*}
	p'(z)<(\ga-1)\ga -4+2\ga b_1+\frac{(m+2-\ga)(\ga+1+2b_1)}{2} = \frac{\ga^2+(m-1)\ga+m-6}{2} + (2+m+\ga)b_1<0
\end{align*}
for any $\ga\in(1,\gau)$. Hence
\begin{align}\label{p(z)>p(zs)>0}
	p(z)\geq p(z_0(\ga))> p(z_0(1)) >0,
\end{align}
where the last inequality is shown in Lemma \ref{p(zs)>0}. Therefore, by using $r_2-r_1<0$, we obtain
\begin{align}\label{W<0}
	b_3 <0.
\end{align}
Combining the inequalities \eqref{W<0} and $V_8<r_2<0$ with the simple fact that the function $f(V)=\frac{V-r_2}{r_1-V}$  is  increasing for $V\in(r_2,0)$, we have derived the inequality. 

It remains to prove the second inequality of \eqref{phi2}. Since $b_2+b_3 = -\frac{1}{2}(\ga+1+2b_1)$, we see easily from~\eqref{B} and the inequality $z<\frac15$, that
\begin{align*}
	0>b_2+b_3 >-\frac{\ga+1}{2} > -\frac{3}{2}.
\end{align*}
As the function $ \frac{r_1-V}{r_1-V_8}$ is decreasing in $V$, we have
\begin{align}\label{3.10-2}
	(\frac{r_1-V}{r_1-V_8})^{b_2+b_3} < (\frac{r_1}{r_1-V_8})^{-\frac{3}{2}}.
\end{align}
We claim that $\frac{r_1-V_8}{r_1}<\frac{9}{7}$. By using \eqref{r1} and $\frac{d V_8}{dz}<0$ (see \eqref{ineq: V'6>0, V'8<0}), we have
\begin{align*}\label{3.10-3}
	-\frac{2r_1}{7}-V_8 <  \frac{-4+m[(20-11\ga)z+7-2\sqrt 3-7w(z)]}{14m}
	\end{align*}
We will show $q(z;m)<0$. If $(20-11\ga)z+7-2\sqrt 3-7w(z)<0$, then we are done. If $(20-11\ga)z+7-2\sqrt 3-7w(z)>0$, then $q(z;1)<q(z;2)$. Hence, it is enough to check $q(z;2)<0$. Since $w'(z)<0$ and $20-11\ga>0$ for any $\ga\in(1,\gau)$, $q'(z;2)>0$. It is therefore sufficient to check $q(z_0;2)<0$, and this is shown in Proposition~\ref{prop:poly} \eqref{q(z0)<0}.
We obtain
\begin{align*}
	(\frac{r_1-V}{r_1-V_8})^{b_2+b_3} < (\frac{7}{9})^{-\frac{3}{2}}<\frac{3}{2}.
\end{align*}
This completes the proof of \eqref{phi2}. Thus, \eqref{ineq:U(V)<-V} holds for  $V_b\in(V_8,0)$ which contradicts \eqref{assumption}. It follows that $V_b=0$. This completes the proof.
\end{proof}

The next two lemmas provide useful technical properties that will enable us to conclude the existence of a suitable lower barrier for the solution in the fourth quadrant in the case $\ga\in(1,\gau)$, $z\in(z_g,z_0)$.

\begin{lemma}[Lower bound of $z_g$]\label{func lower bound of zg}
Let $m=1, 2$. For any $\ga\in(1,\gau)$, $\frac{2(\ga-1)}{11}<z_g(\ga)$.
\end{lemma}

\begin{proof}
Recalling \eqref{zg}, for any $\ga\in(1,\gau)$, when $m=1$, we have
\begin{align*}
	z_g(\ga) = \frac{\gamma-1}{\gamma(\sqrt{\gamma^2+(\gamma-1)^2}+\gamma)} >  \frac{\gamma-1}{\gau(\sqrt{\gau^2+(\gau-1)^2}+\gau)} >\frac{2(\ga-1)}{11}.
\end{align*}
When $m=2$, since 
\begin{align*}
	z_g(\ga) = \frac{2(\ga-1)}{\sqrt{(2\ga^2-\ga+1)^2+2\ga(\ga-1)[4\ga(\ga-1)+\frac{8}{3}]}+2\ga^2-\ga+1}=:\frac{2(\ga-1)}{f(\ga)},
\end{align*}
it is enough to show that $f(\ga)<11$ for any $\ga\in(1,\gau)$. We compute
\begin{align*}
	f'(\ga) = \frac{(72\ga^2-90\ga+44)\ga+11(\ga-1)}{\sqrt{3(1-3\ga)^2(4\ga^2-4\ga+3)}}+4\ga-1>0.
\end{align*}
Hence, for any $\ga\in(1,\gau)$,
\begin{align*}
	f(\ga)< f(\gau)<11.
\end{align*}
This completes the proof.
\end{proof}

\begin{lemma}[Upper bound of $C_8$]\label{Upper bound of C8}
Let $m=1,2$. For any $\ga\in(1,\gau)$ and $z\in(z_g,z_0)$, $C_8\leq 1+2(\ga-2)z$.
\end{lemma}

\begin{proof}
Direct computation gives
\begin{align*}
	C_8-1-2(\ga-2)z = \frac{-1-3(\ga-2)z+w(z)}{2} =- \frac{4[(\ga-2)^2z^2+(\ga-1)z]}{(1+3(\ga-2)z+w(z))} <0, 
\end{align*}
for any $\ga\in(1,\gau)$ and $z\in(z_g,z_0)$. Hence, $C_8\leq 1+2(\ga-2)z$.
\end{proof}

The next lemma provides the final barrier argument we require to conclude that, in the remaining case, $\ga\in(1,\ga_u)$, $z\in(z_g,z_0)$, the solution trajectory passes below $C_g$ before meeting the lower sonic line.
\begin{lemma}\label{m2 small ga and z}
Consider $P_*=P_6$ and $m=1,2$. For any $\ga\in (1,\gau)$ and $z\in (z_g(\ga),z_0(\ga))$, $B(C)=-C(1+C)$ is a lower barrier of solution trajectories $V=V(C)$ provided by Lemma \ref{pass collapse} for $C\in(C_9,0)$ in the 4th quadrant.
\end{lemma}

\begin{proof}
Clearly, the curve $V=B(C)$ is tangential to the curve $V=-C$ at $C=0$, while, from Lemma \ref{upper bound of sol}, the solution trajectory $V(C)$ satisfies $V'(0)>-1$. Consequently, there exists $\epsilon>0$ such that the solution $V(C)>B(C)$ for $C\in(-\epsilon,0)$. By applying the barrier argument, to conclude the proof, it is sufficient to show that 
\begin{align}
	\frac{G(-C(1+C),C;\ga,z)}{F(-C(1+C),C;\ga,z)}+(1+2C) <0 
\end{align}
for any $\ga\in(1,\gau)$, $z\in(z_g,z_0)$ and $C\in(C_9,0)$. As $F(-C(1+C),C;\ga,z)>0$ by Lemma \ref{L:FG=0} and $1-C(1+C)>0$ for $C\in[C_9,0]$, it is equivalent to show that
\begin{align}
	\frac{C^2}{2}\Barrier(C;m):=\Big[G(-C(1+C),C;\ga,z)+(1+2C)F(-C(1+C),C;\ga,z) \Big]\Big(1-C(1+C)\Big)<0 .
\end{align}
A direct computation shows
\begin{align}\label{6th poly}
	\Barrier(C;m)=&\,2[m(\ga-1)+1]C^6+[7m(\ga-1)+6]C^5+[5m\ga-3m-2+2m(\ga-2)\ga z]C^4\notag\\&+[m(5\ga-9)(\ga z-1)-12]C^3+[5m(1-\ga)+2+2m\ga^2 z]C^2\notag\\&+ [m(\ga-3)+6-m(2\ga^2-6\ga+2)z]C+m(\ga-1)-2-m(\ga^2+\ga-4)z.
	\end{align}
We shall show that $\Barrier(C;m)<0$ for any $m=1,2$, $\ga\in(1,\gau)$, $z\in(z_g,z_0)$ and $C\in[C_9,0)$ by proceeding in two steps, which, when combined, show $\Barrier(C;m)<0$ for $C\in[C_9,0)$:
\begin{enumerate}
\item $\Barrier(C;m)$ has only  one zero in the interval $C\in(-1,0)$ for any $\ga\in(1,\gau)$ and $z\in(z_g,z_0)$.
\item $\Barrier(0;m)<0$, $\Barrier(C_9;m)<0$.
\end{enumerate}
The first step follows from the Fourier-Budan Theorem (see Theorem \ref{Fourier-Budan Theorem}) and is proven in Lemma \ref{one zero-1}. For the second step, it is clear that $\Barrier(0;m) = m(\ga-1)-2-m(\ga^2+\ga-4)z<0$ and $\Barrier(-1;m)=2m(3-\ga)z>0$. From Lemma \ref{Upper bound of C8}, therefore, to show that $\Barrier(C_9;m)<0$, it is sufficient to verify that $\Barrier(-1-2(\ga-2)z;m)<0$. We compute
\begin{align*}
	\Barrier(-1-2(\ga-2)z;m)
	=2z[mA(\ga,z)+B(\ga,z)],
\end{align*}
where
\begin{align*}
	A(\ga,z) &= 64(\ga-2)^6(\ga-1)z^5+16(\ga-2)^5(6\ga-5)z^4+4(\ga-2)^3(3\ga^2-3\ga-8)z^3\\&\quad-2(\ga-2)^2(11\ga^2-41\ga+36)z^2-(\ga-2)(\ga^2-5\ga-2)z+(\ga-3)^2,\\
	B(\ga,z)&=64(\ga-2)^6z^5+96(\ga-2)^5z^4-16(\ga-2)^4z^3-64(\ga-2)^3z^2-8(\ga-2)^2z+8(\ga-2).
\end{align*}
We first claim that $A(\ga,z)>0$. We check
\begin{align*}
\frac{dA(\ga,z)}{dz}	&=64(\ga-2)^5z^3[5(\ga-2)(\ga-1)z+6\ga-5]+12(\ga-2)^3(3\ga^2-3\ga-8)z^2\\&\quad-4(\ga-2)^2(11\ga^2-41\ga+36)z
	-(\ga-2)(\ga^2-5\ga-2)\\
	&<64(\ga-2)^5(\ga^2+3\ga-3)z^3+\frac{12}{25}(\ga-2)^3(3\ga^2-3\ga-8)+\frac{8}{5}(\ga-2)^2-(\ga-2)(\ga^2-5\ga-2)\\
	&<\frac{\ga-2}{25}(36\ga^4-180\ga^3+167\ga^2+405\ga-414)<0,
\end{align*}
where we have used $z<z_0<z_M<\frac{1}{5}$, the fact that $5(\ga-2)(\ga-1)z+6\ga-5$ is decreasing as a function of $z$, and that $11\ga^2-41\ga+36>-2$ for any $\ga$ in the first inequality, and the negativity of the last quantity is shown in Proposition~\ref{prop:poly} \eqref{L 3.17 p1}. Since $\frac{dA(\ga,z)}{dz}<0$ is valid for any $z<\frac{1}{5}$, we have $A(\ga,z)>A(\ga,\frac{1}{5}) >0$ 
for any $\ga\in(1,\gau)$ and $z\in(z_g,z_0)$, where
$A(\ga,\frac{1}{5})>0$ is shown in Proposition~\ref{prop:poly} \eqref{L 3.17 p2}. Consequently, $\Barrier(-1-2(\ga-2)z;1)< \Barrier(-1-2(\ga-2)z;2)$.
It is then sufficient to check $\Barrier(-1-2(\ga-2)z;2)<0$. We denote $p(z)=\frac{\Barrier(-1-2(\ga-2)z;2)}{4z}$, and a direct computation gives
\begin{align}\label{p(z)}
	p(z)=&\,32(2\ga-1)(\ga-2)^6z^5+32(3\ga-1)(\ga-2)^5z^4+4(3\ga^2-5\ga-4)(\ga-2)^3z^3\notag\\
	&-2(11\ga^2-25\ga+4)(\ga-2)^2z^2-(\ga^2-\ga-10)(\ga-2)z+(\ga-1)^2.
\end{align}
Our objective now is to show that $p(z)<0$. First,
\begin{align*}
	p'(z)
	= 4(\ga-2)(f(z)+g(z))
\end{align*}
where 
\begin{align*}
	f(z)&=40(2\ga-1)(\ga-2)^5z^4+32(3\ga-1)(\ga-2)^4z^3,\\
	g(z)&=3(3\ga^2-5\ga-4)(\ga-2)^2z^2-(11\ga^2-25\ga+4)(\ga-2)z-\frac{\ga^2-\ga-10}{4}.
\end{align*}
We claim that both $f(z)$ and $g(z)$ are positive for any $\ga\in(1,\gau)$ and $z\in(z_g,z_0)$. We check $f(z)>0$ first. Since $(2\ga-1)(\ga-2)<0$ for any $\ga\in(1,\gau)$, and $z<z_M<\frac{1}{5}$, 
\begin{align*}
	f(z) = 8(\ga-2)^4z^3[ 5(2\ga-1)(\ga-2)z+4(3\ga-1)] &> 8(\ga-2)^4z^3[ (2\ga-1)(\ga-2)+4(3\ga-1)] >0. 
\end{align*}
For $g(z)$, we notice that $3(3\ga^2-5\ga-4)(\ga-2)^2<0$, $-(11\ga^2-25\ga+4)(\ga-2)<0$ and $z_0(\ga) = \frac{22-5\ga}{125}<\frac{3}{20}$ for any $\ga\in(1,\gau)$. Hence, $g(z)>g(\frac{3}{20})>0$  
where the positivity of
$g(\frac{3}{20})$ is proven in Proposition~\ref{prop:poly} \eqref{L 3.13 p3}. Therefore, we have $p'(z)<0$ since $\ga-2<0$. 
It is then sufficient to check $p(\frac{2(\ga-1)}{11})<0$ by Lemma \ref{func lower bound of zg}. 
The negativity of $p(\frac{2(\ga-1)}{11})$ is proven in Proposition~\ref{prop:poly} \eqref{p(2(ga-1)/11)<0}. 

In conclusion, we have established that $p(z)<0$  for any $\gamma\in(1,\gau)$ and $z\in(z_g,z_0)$. Consequently, $\Barrier(-1+2(2-\gamma)z)<0$, implying $\Barrier(C)\neq 0$ for $C\in(-1+2(2-\gamma)z,0)$. Therefore, $\Barrier(C)<0$ for any $C\in(C_9,0)$. This completes the proof.
\end{proof}

We now have all the necessary ingredients to prove the main result of this section.

\begin{proposition}\label{approach C=-1-V}
Consider $P_*=P_6, P_8$, $m=1,2$. For any $\ga\in\Gamma(P_*)$ and $z\in\mathring{\mathcal{Z}}(\ga;P_*)$, 
the solution trajectories provided by Lemma \ref{pass collapse} approach the lower sonic line $C=-1-V$ at $(-C_s-1,-C_s)$ with $C_s<\Cg$.
\end{proposition}
\begin{proof}
This follows directly from Lemmas \ref{barrier P8} and \ref{barrier P5} for $P_*=P_8$ and from Lemma \ref{lower barrier when m=1,2}, Lemma \ref{m2 small ga and z}, and the uniqueness of the ODE solutions when $P_*=P_6$.
\end{proof}

\section{Jump Locus and existence of $P_H$}\label{sec:jump}

Our remaining goal is to connect the solution curve a unique solution curve from $P_\infty$, defined in \eqref{Pinfty}.
We first show that there exists a unique solution curve emanating from $P_\infty$. 
\begin{lemma}[The solution from $P_\infty$]\label{P8-P_infity}
Let $m=1,2$, $\ga\in(1,3]$ and $z\in (0,z_M(\ga)]$. 
There exists $C_0<0$ small enough and a unique solution curve $V_\infty(C)$ to \eqref{ODE} with
\begin{align}
	\lim_{C\to \Cg^-}V_\infty(\Cg)=\Vg,\quad \lim_{C\to -\infty} V_\infty(C) = \overline V_\infty.
\end{align}
Moreover, $V_\infty'(C)<0$ for $C\in(-\infty, \Cg)$.
\end{lemma}

\begin{proof}
	We decompose 
	\begin{align*}
		V(C) = \overline V_\infty +\tilde{V}(C)
	\end{align*}
	where $\lim_{C\to -\infty} \tilde{V}(C)=0$. 
Denote $\alpha = 1+\frac{mz}{1+\overline V_\infty}$. Then, we see that $\tilde{V}(C)$ satisfies
	\begin{align*}
	\tilde V'(C)-\frac{(m+1)\tilde V(C)}{\alpha C}
	&=\frac{m+1}{C} K(\tilde V, C)
	\end{align*} 
	where 
	\[
	K(\tilde V, C):=  \frac{-\frac{V(V+1)(V+1+m\ga z)}{(m+1)C^2}+\frac{a_1(1+V)^2-a_2(1+V)+a_3}{\alpha C^2}\tilde V+\frac{mz}{\alpha (1+ \overline V_\infty)(1+V)}\tilde V^2}{1+\frac{mz}{1+V}-\frac{a_1(1+V)^2-a_2(1+V)+a_3}{C^2}}. 
	\]
Therefore,
\begin{align}\label{integrated ode}
	(\frac{\tilde V}{(-C)^{(m+1)\alpha ^{-1}}})' 
	& =- \frac{m+1}{(-C)^{1+(m+1)\alpha ^{-1}}}K(\tilde V, C), 
\end{align}
and so we solve
\begin{align}\label{tildeV}
	\tilde V(C)=-(-C)^{(m+1)\alpha ^{-1}}\int_{-\infty}^{C} 
	\frac{(m+1)}{(-\hat{C})^{1+(m+1)\alpha ^{-1}}} K (\tilde V, \hat C)
	d\hat{C}.
\end{align}
By  the contraction mapping theorem, there exists a unique solution $\tilde V(C)$ which solves \eqref{tildeV} such that
\begin{align}\label{bound of tildeV}
	|\tilde V(C)| \lesssim \frac{1}{C^2},\quad \text{ for } |C|>|C_0|
\end{align}
for a sufficiently large  $|C_0|$. Now, let $V_\infty(C)$ be a solution to \eqref{integrated ode} on $(-\infty,C_0]$ with $\lim_{C\to -\infty} V_\infty(C) = \overline V_\infty$. We integrate \eqref{integrated ode}. The integral converges at $-\infty$ from the priori bound $|V|\lesssim 1$, hence $\tilde V$ satisfies \eqref{tildeV} and the a priori bound \eqref{bound of tildeV}, and hence the solution is unique. Next, we extend the unique solution to $C\in(-\infty,\Cg)$.  We claim that $V_G(C)< V_\infty(C)< V_F^+(C)$ for $C<\Cg$.
From Lemma \ref{L:FG=0}, we have
\begin{align}
	D(V,C)<0, \quad F(V,C)<0, \quad G(V,C)>0,
\end{align}
for any $(V,C)\in \{(V,C): V_G(C)< V< V_F^+(C), \ C\in(-\infty,\Cg)\}$. 

We first show that $V_G(C)<V_\infty(C)<V_F^+(C)$ for $C\in(-\infty,C_0)$ for some $|C_0|$ sufficiently large. The inequality $V_G(C)<V_\infty(C)$ holds as $V'_\infty(C)>0$ when $V_G(C)>V_\infty(C)$, and thus if there exists $\tilde C\in(0,C_0)$ such that $V_\infty(C)<V_G(C)$, we cannot have $V(C)\to\overline{V_\infty}$ as $C\to-\infty$. 
In addition, $V(C)<V_F^+(C)$ holds for large $|C|$ as $V_\infty(C) \to \overline{V}_\infty$ and $V_F^+(C)\to \infty$ as $C\to-\infty$  from Lemma \ref{L:FG=0}.

 By the definition of $\Pg=(\Vg,\Cg)$ (see \eqref{Pg}), 
the two curves $V_F^+(C)$ and $V_G(C)$ intersect at $\Pg$. 
In order to show that $\lim_{C\to \Cg^-} V_\infty (C) = \Vg$, it is therefore sufficient to show that $V_\infty(C)$ does not intersect $V_G$ or $V_F^+$ for $C<\Cg$.

\underline{Case I:} Suppose that $V_\infty(C)$ first intersects $V=V_G(C)$ at $C_r\in(-\infty,\Cg)$. Then we have $V_\infty(C)>V_G(C)$ for any $C<C_r$, and $V_\infty'(C_r)=0$.
As $V_G'(C)<0$ by Lemma \ref{L:FG=0}, it follows that $V_G'(C)<V_\infty'(C)$ and $V_G(C)>V_\infty(C)$ for $C\in(C_r-\epsilon, C_r)$, where $\epsilon$ is sufficiently small. This is a contradiction to $V_G(C)<V_\infty(C)$ for $C<C_r$.

\underline{Case II:} Suppose that $V_\infty(C)$ first intersects  $V=V_F^+(C)$ at $C_r\in(-\infty,\Cg)$. Then 
$V_\infty(C)<V_F^+(C)$ for any $C<C_r$, and $V_\infty'(C_r)=-\infty.$ 
This is again a contradiction since $(V_F^+)'(C)<0$ by Lemma \ref{L:FG=0} and $V_\infty(C)<V_F^+(C)$ for any $C<C_r$. 

Therefore, $V_\infty(C)$ can only exit the region at $\Pg$. Since there is no other accumulation point in this region, the solution curve must approach $\Pg$ as $C\uparrow \Cg$. 
\end{proof}

\begin{remark}
The solution from $P_\infty$, considered as a pair of functions $(V_\infty(x),C_\infty(x))$ via the ODE~\eqref{ODE system} is easily seen to exist on a half-infinite interval of $x$, $[\mathring{x},\infty)$ with $\lim_{x\to\infty}(V_\infty(x),C_\infty(x))=P_\infty$ via standard ODE arguments, see, for example, \cite{Jenssen18,Lazarus81} for details. Moreover, one finds $V_\infty(x)=\overline{V}_\infty+O(x^{-2\sigma})$, $C_\infty(x)=\overline{C}_\infty x^{\sigma}+O(x^{-\sigma})$, for $\sigma=\frac{1}{\la}\big(1+\frac{mz}{1+\overline{V}_\infty}\big)$. In particular, the trajectory gives rise to a solution in the $(t,r)$-plane which covers the whole domain $r>0$.
\end{remark}

By Proposition \ref{approach C=-1-V}, the solution trajectories provided by Lemma \ref{pass collapse} must approach $C = -1-V$ at a point $(-C_s - 1,C_s)$ which lies strictly below the triple points $P_7$ and $P_9$. To construct a global solution, we consider the Rankine-Hugoniot conditions \eqref{jump condition} for the solution curves to cross $C = -1-V$. The next lemma presents important properties of the jump conditions.

\begin{figure}
\centering
	\begin{minipage}[b]{.4\textwidth}
	\centering
		\includegraphics[scale = 0.3]{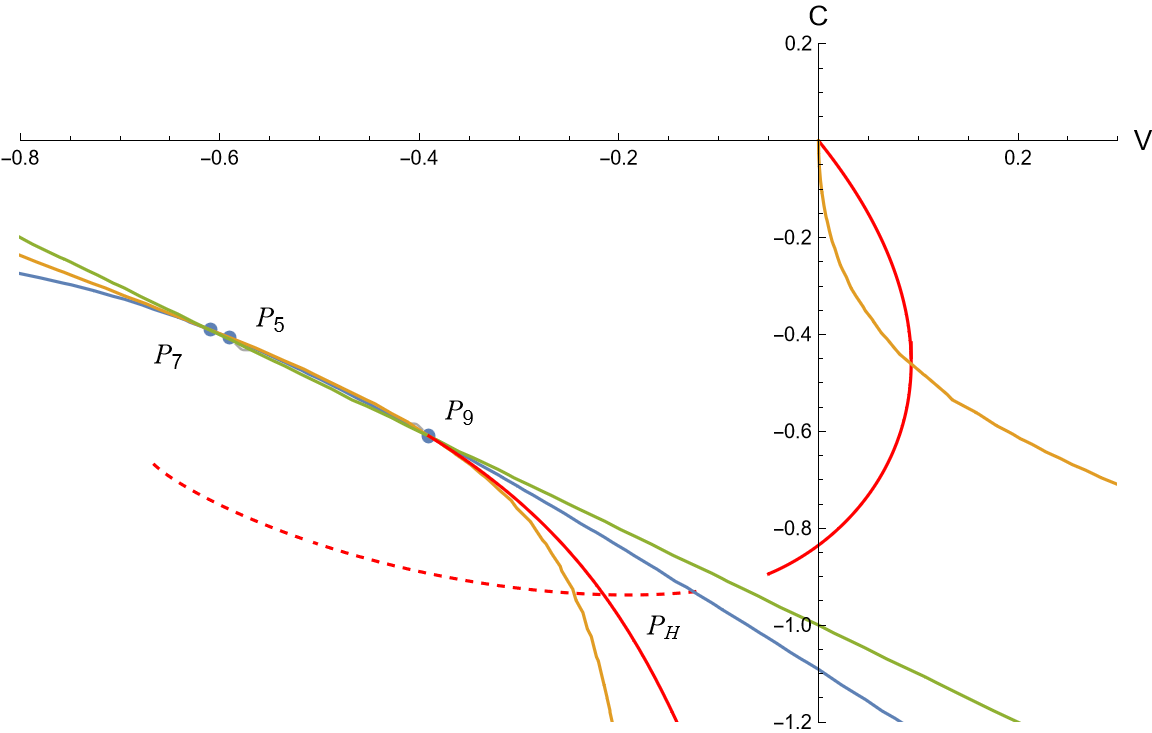}
		\caption{$\ga=2$, $z=0.119$, $m=2$}
	\end{minipage}
	\qquad
	\begin{minipage}[b]{.4\textwidth}
	\centering
		\includegraphics[scale = 0.3]{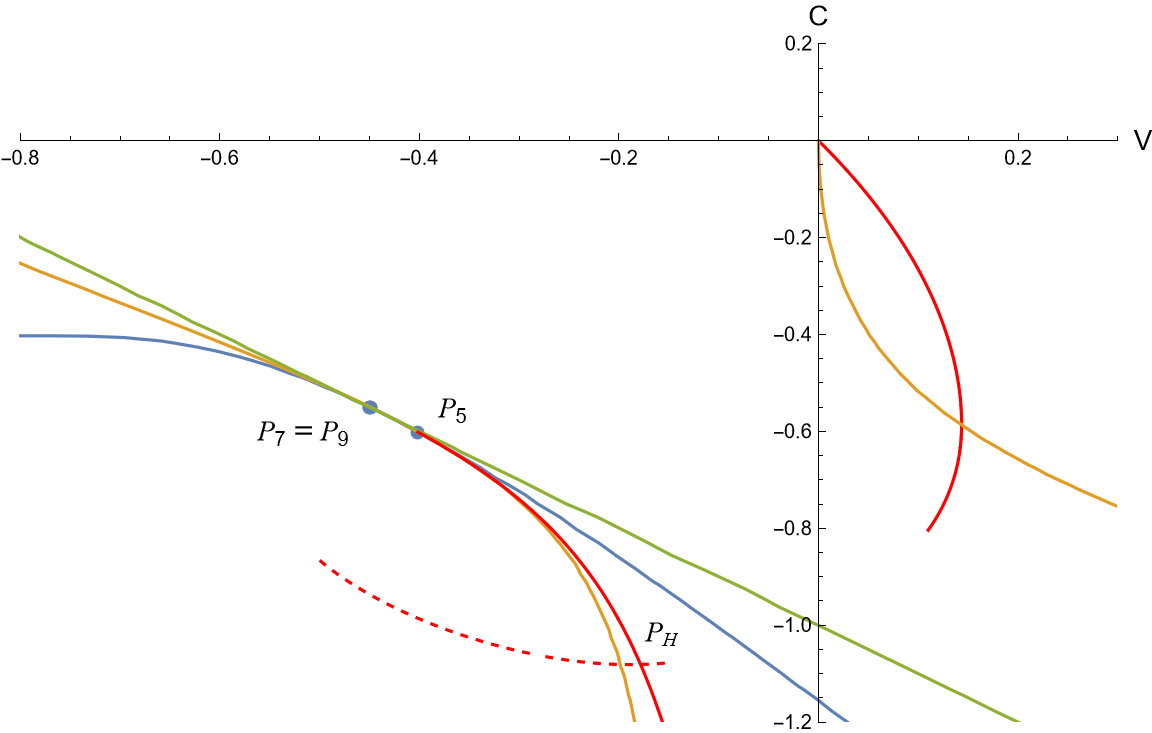}
		\caption{$\ga=3$, $z=z_M(3)$, $m=2$}
	\end{minipage}
	\vspace{-5mm}
\end{figure}
\begin{lemma}\label{1-1 jump point}
Let $m=1,2$, $\ga\in(1,3]$ and $z\in (0,z_M]$. The Rankine-Hugoniot conditions \eqref{jump condition} establish a one-to-one correspondence between the points $(V_-,C_-)\in S_U$ and $(V_+,C_+)\in S_L.$ In other words, $\Jp(V,C)$, as defined in \eqref{Id: Jp}, is well-defined and bijective from $S_U$ to $S_L$. Moreover, 
$\Jp(V,C)$ satisfies
	\begin{enumerate}
		\item[(i)] For $\kappa\in[0,1)$, the  jump locus of the line $C(V) = -\kappa (1+V)$ is the line $C(V) = -\mathcal I(\kappa)(1+V)$, where with $\mathcal I(\kappa) =  \sqrt{\frac{2\ga-(\ga-1)\kappa^2}{\ga-1+2\kappa^2}} \in (1, \sqrt{\frac{2\ga}{\ga-1}}]$;		
		\item[(ii)] $\Jp(0,0) = (\frac{-2}{\ga+1},-\frac{\sqrt{2\ga(\ga-1)}}{\ga+1}) = (V_1, -C_1)=:\tilde P_1$;
		\item[(iii)] For any $(V,C)\in S_U$, $V>\Jp_1(V,C)$ and $C>\Jp_2(V,C)$.
	\end{enumerate}
\end{lemma}
\begin{proof}
We first show that $\Jp$ is well-defined from $S_U$ to $S_L$. Let $(V,C)\in S_U$. By using \eqref{jump condition}, we obtain
\begin{align*}
    \Jp_2^2(V,C)-(1+\Jp_1(V,C))^2 & = C^2 + \frac{\gamma-1}{2}[(1+V)^2-(1+\Jp_1(V,C))^2]  -(1+\Jp_1(V,C))^2\\
    & = \cfrac{3-\gamma}{1+\gamma}C^2+\cfrac{\gamma-1}{\gamma+1}(1+V)^2-\cfrac{2C^4}{(\gamma+1)(1+V)^2}\\
    &> \cfrac{3-\gamma}{1+\gamma}C^2+\cfrac{\gamma-1}{\gamma+1}(1+V)^2-\cfrac{2}{(\gamma+1)}C^2 = \cfrac{\gamma-1}{\gamma+1}[(1+V)^2-C^2]>0.
\end{align*}
Also, it is easy to check that 
\begin{align*}
    \frac{\ga-1}{2\ga}\Jp_2^2(V,C)-(1+\Jp_1(V,C))^2    & =-\frac{\ga-1}{2\ga}C^2-\cfrac{C^4}{\ga(1+V)^2}\leq 0
\end{align*}
where the equality holds when $C=0$. Hence, $\Jp$ is a well-defined map from $S_U$ to $S_L$.
Next, we show that $\Jp$ is injective. Suppose that there are $(V,C)$ and $(\tilde V,\tilde C)$ such that the corresponding jump points  $\Jp(V,C)$ and $\Jp(\tilde V,\tilde C)$ are same. By using \eqref{jump condition}, we obtain
\begin{align*}
	(1+\Jp_1(V,C))(V-\tilde V) = \frac{2}{\ga+1}\Big[\frac{\ga-1}{2}[(1+V)^2-((1+\tilde V)^2)]+C^2-\tilde C^2\Big]=0.
\end{align*}
Since $\Jp_1(V,C)\neq -1$, it follows that $(V,C)=(\tilde V,\tilde C)$. 
$(i)$ and $(ii)$ follow by direct computation and imply that $\Jp$ is also surjective, hence $\Jp$ is bijective from $S_U$ to $S_L$. For $(iii)$, by using \eqref{jump condition}, we obtain
	\begin{align*}
		\Jp_1(V,C)- V = (1+\Jp_1(V,C))- (1+V) = \frac{2(1+V)}{\ga+1} [\frac{C^2}{(1+V)^2}-1]<0
	\end{align*}
	where we used $(V,C)\in S_U$ in the last inequality. Together with \eqref{jump condition}, this also implies
	\begin{align*}
		\Jp_2^2(V,C)-C^2 = \frac{\ga-1}{2}[(1+V)^2-(1+\Jp_1(V,C))^2]>0. 
	\end{align*}\end{proof}

Before we show that the jump locus of the solution trajectory through the origin connects to $V_\infty(C)$, we first require an understanding of the location of $P_5$ relative to the jump locus. This is the content of the next lemma.
\begin{lemma}\label{pre jump of P_5}
	For $\ga\in[\ga_g,3]$ and $z\in(z_g,z_M]$, $P_5=(V_4,C_5)$ is above the jump locus $\Jp( -\frac{2}{3}C^2,C)$ of the barrier function $B(C) = -\frac{2}{3}C^2$ in the $C$ coordinates.
\end{lemma}
\begin{proof}
	We first note 
	that the barrier  $V=B(C) = -\frac{2}{3}C^2$ can also be viewed as $C =-\sqrt{ -\frac{3}{2}V}$. Since $C^2=-\frac{3}{2}V$ and $C = -(1+V)$ intersect at the point $(\frac{\sqrt{33}-7}{4},\frac{3-\sqrt{33}}{4})$, in order to prove the statement, it is sufficient to show that $C_5^2<\Jp_2^2(V,-\sqrt{-\frac{3}{2}V})$ in the interval $V \in (\frac{\sqrt{33}-7}{4},0)$. 	 Recalling \eqref{Id: Jp},
	 \begin{align*}
	 	\Jp_2^2(V,-\sqrt{-\frac{3}{2}V}) = \frac{3(\ga^2-6\ga+1)V}{2(\ga+1)^2} +\frac{2\ga(\ga-1)(1+V)^2}{(\ga+1)^2}-\frac{9(\ga-1)V^2}{2(\ga+1)^2(1+V)^2}.
	 \end{align*}
	 We check
	 \begin{align*}
	 \frac{d \Jp_2^2(V,-\sqrt{-\frac{3}{2}V})}{dV} &= \frac{3(\ga^2-6\ga+1)}{2(\ga+1)^2} +\frac{4\ga(\ga-1)(1+V)}{(\ga+1)^2}-\frac{9(\ga-1)V}{(\ga+1)^2(1+V)^3}\\
	 &>\frac{(11\ga^2-26\ga+3)(1+V)-3(\ga^2-5)V}{(\ga+1)^2}>0. 
	 \end{align*}
	 It follows that $ \Jp_2^2(V,-\sqrt{-\frac{3}{2}V})$ is increasing in the interval $V \in (\frac{\sqrt{33}-7}{4},0)$. Hence, $ \Jp_2^2(V,-\sqrt{-\frac{3}{2}V})>(\frac{3-\sqrt{33}}{4})^2>\frac{4}{9}>C_5^2$ by using Lemma \ref{C4<2/3}.
	 \end{proof}
In the following, we use $V=V_{\text{sol}}(C)$ to denote the solution curve initiated from $P_*$ and propagated by \eqref{ODE}, and recall $P_s=(V_s,C_s)$ is the intersection point of $V = V_{\text{sol}}(C)$ and $C = -1-V$.

\begin{lemma}[Existence of $P_H$]\label{PH}
Let $P_*$ be either $P_6$ and $P_8$, and $m=1,2$. Let $\ga\in\Gamma(P_*)$ and $z\in \mathring{\mathcal{Z}}(\ga;P_*)$. There exists at least one $P_H\in S_L\cap \{V>V_8\}$ such that the jump locus of the solution trajectory, $\{\Jp(V_{\text{sol}}(C),C)\mid C\in(C_s,0)\}$, intersects  $V=V_\infty(C)$ at $P_H$.
\end{lemma}

\begin{proof}
From Proposition \ref{approach C=-1-V}, $V_{\text{sol}}(C)$ has a unique intersection with $C = -1-V$ at $P_s=(V_s,C_s)$ with $C_s< \Cg$. 
Since $\Jp(V_s,C_s)=(V_s,C_s)$ and $\Jp(0,0) = \tilde P_1$, the jump locus $\Jp(V_{\text{sol}}(C),C)$ connects $\tilde{P}_1$ and $P_s$.
As $\tilde{P}_1$ is the reflection of $P_1$ about the $V$-axis, we have $G(V_1,-C_1)<0$ and $F(V_1,-C_1)<0$. Moreover, the fact that $C_s<\Cg$ and $P_s$ lies on $C = -1-V$ implies $G(V_s,C_s)>0$ and $F(V_s,C_s)>0$ by Lemma \ref{L:FG=0}.  Hence, $\Jp(V_{\text{sol}}(C),C)$ intersects both the curves $V_G$ and $V_F^+$. From Lemma \ref{P8-P_infity}, $V_\infty$ is in between $V_G$ and $V_F^+$ in the interval $C\in(-\infty, \Cg)$. Given the definition of $\Pg$ in \eqref{Pg} and together with Lemma \ref{1-1 jump point}, it suffices to show that $\Jp(V_{\text{sol}}(C),C)$ intersects both $V_G$ and $V_F^+$ at points with $C$ values less than $\Cg$. We will discuss the cases separately for $\Pg=P_9$ and $\Pg=P_5$.

When $\Pg=P_9$, the three curves $C = -(1+V)$, $V_F^+$ and $V_G$ intersect at $P_9$. It follows that $\Jp(V_{\text{sol}}(C),C)$ intersects $V_\infty$ at some points $P_H$ with $C_H<C_9$ by Lemma \ref{1-1 jump point}. 

In the case $\Pg = P_5$, by Lemma \ref{pre jump of P_5}, $P_5$ is above $\Jp(V_{\text{sol}}(C),C)$ in the $C$ coordinate in the interval $C\in(C_5,0)$. For  $C\in(C_s,C_5)$, we apply Lemma \ref{1-1 jump point}\textit{(iii)} to conclude that the part of $\Jp(V_{\text{sol}}(C),C)$ in the interval $C\in(C_s,C_5)$ cannot cross the line $C=C_5$. Hence,  $\Jp(V_{\text{sol}}(C),C)$ intersects $V_\infty$ at some point $P_H$ with $C_H<C_5$.
\end{proof}

\begin{proposition}\label{xs}
	Let $P_*$ be either $P_6$ and $P_8$, and $m=1,2$. Let $\ga\in\Gamma(P_*)$ and $z\in \mathring{\mathcal{Z}}(\ga;P_*)$. There exists $x_H\in(0,\infty)$ such that at $x=x_H$, the solutions $(V(x),C(x))$ to \eqref{ODE system}, provided by Theorem \ref{approach C=-1-V}, jump to $(V_\infty(\Jp_2(V(x_H),C(x_H))),\Jp_2(V(x_H),C(x_H)))$ by \eqref{jump condition}. 
\end{proposition}

\begin{proof}
  This follows directly from Lemma \ref{PH}. The finiteness of $x_H$ follows from  standard ODE theory.
\end{proof}

Next, we show that $P_H$ is unique when $\ga\in(1,\gaSix]$ and $z\in (z_g(\ga),z_M(\ga)]$. We first derive some properties of the solution curve $V_{\text{sol}}(C)$ and the jump locus of the solution curve $\Jp(V_{\text{sol}}(C),C)$. 

\begin{lemma}\label{1 intersection}
	Consider $P_*=P_6$, and $m=1,2$. Let $\ga\in(1,\gaSix]$ and $z\in (z_g(\ga),z_M(\ga)]$. For each $\kappa\in(0,1)$, the line $C=-\kappa(1+V)$ intersects $\{(V_{\text{sol}}(C),C)\,|\,C\in(C_s,0)\}$ at a unique point. Similarly, the line $C=-\mathcal I(\kappa)(1+V)$, as defined in Lemma \ref{1-1 jump point},  intersects $\{\Jp(V_{\text{sol}}(C),C)\,|\,C\in(C_s,0)\}$ at a unique point.
\end{lemma}
\begin{proof}
	The second part of the statement follows from the first part and the bijectivity of the jump map $\mathcal{J}$ between $S_U$ and $S_L$ as shown in Lemma \ref{1-1 jump point}. Therefore, we only need to show the first part of the statement.
	
	Since $V_{\text{sol}}(0)=0$, 
	and $V_{\text{sol}}(C_s)=V_s$ lies on $C=-(1+V)$, $V_{\text{sol}}(C)$ intersects $C=-\kappa(1+V)$ at least once for each $\kappa\in(0,1)$. In order to show the intersection is unique, it is equivalent to show that the function $\frac{C}{1+V_{\text{sol}}(C)}$ is monotone increasing with respect to $C$. We compute
	\begin{align*}
		\frac{d}{dC} (\frac{C}{1+V_{\text{sol}}(C)}) = \frac{(1+V_{\text{sol}}(C))F(V_{\text{sol}}(C),C;\ga,z)-CG(V_{\text{sol}}(C),C;\ga,z)}{F(V_{\text{sol}}(C),C;\ga,z)(1+V_{\text{sol}}(C))^2}.
	\end{align*}
	Since $F(V_{\text{sol}}(C),C;\ga,z)>0$ for any $C\in(C_s,0)$ and $G(V_{\text{sol}}(C),C;\ga,z)\geq 0$ in the interval $C\in[C_g,0)$ by Lemma \ref{L:FG=0}, 
	it is clear that $\frac{d}{dC} (\frac{C}{1+V_{\text{sol}}(C)})>0$ in the interval $C\in[C_g,0)$. In the case  $C\in(C_s,C_g)$, $G(V_{\text{sol}}(C),C;\ga,z)< 0$, recalling \eqref{G(V,C)} and \eqref{F(V,C)}, so we have
	\begin{align*}
		\frac{d}{dC} (\frac{C}{1+V_{\text{sol}}(C)})F(V_{\text{sol}}(C),C;\ga,z)(1+V_{\text{sol}}(C))^2=:-\frac{C}{2} p(V_{\text{sol}}(C),C),
	\end{align*}
	where
	\begin{align}
		 p(V,C) = 2(mV+mz-1)C^2 + (1+V)[m(\ga-1)V^2+(2+m(\ga-1)(1-\ga z))V+2+2m\ga z].
	\end{align}
	Since $z\leq z_M = \frac{1}{2+\ga+2\sqrt{2\ga}}$, $1-\ga z\geq 1-\ga z_M = \frac{2+2\sqrt{2\ga}}{2+\ga+2\sqrt{2\ga}}>0$. It is then clear that $p(V,C)>0$ when $V_{\text{sol}}\geq \frac{1-mz}{m}$.
	In the case  $V_{\text{sol}}< \frac{1-mz}{m}$, given that $C^2 < (1+V_{\text{sol}})^2$, we obtain
	\begin{align*}
		p(V_{\text{sol}},C) > p(V_{\text{sol}},-(1+V_{\text{sol}}))
		=m(1+V_{\text{sol}})(\ga+1)[V_{\text{sol}}^2+(1+(2-\ga)z)V_{\text{sol}}+2z].
	\end{align*}
	We claim $p(V_{\text{sol}},-(1+V_{\text{sol}}))>0$. First, the polynomial $q(V) = V^2+(1+(2-\ga)z)V+2z$ is increasing when $V>-\frac{1+(2-\ga)z}{2}$. Moreover, from Proposition \ref{approach C=-1-V}, we have $V_{\text{sol}}>V_s>V_8 = \frac{-1+(\ga-2)z+w}{2}>-\frac{1+(2-\ga)z}{2}$. Hence,  together with $q(V_8)=0$, we derive
	\begin{align}\label{monotone to lines}
		\frac{d}{dC} (\frac{C}{1+V_{\text{sol}}(C)}) >0 , 
	\end{align}
	for any $C\in(C_s,0)$. This completes the proof.
\end{proof}

We are now ready to prove that $P_H$ is unique. 
\begin{lemma}[Uniqueness of $P_H$]\label{uniq PH}
Consider $P_*=P_6$, and $m=1,2$. Let $\ga\in(1,\gaSix]$ and $z\in (z_g(\ga),z_M(\ga)]$.  Then there exists a unique $P_H\in S_L\cap \{V>V_8\}$ such that $\Jp(V_{\text{sol}}(C),C)$ intersects $V_\infty$ at $P_H$.
\end{lemma}
\begin{proof}
We argue by contradiction. Suppose $\Jp(V_{\text{sol}}(C),C)$ intersects $V_\infty$ at two different points $P_{H_1}=(V_{H_1},C_{H_1})$ and $P_{H_2}=(V_{H_2},C_{H_2})$. By Lemma \ref{P8-P_infity},  we have 
 $V_\infty'(C)<0$ within the interval $C\in(-\infty,\Cg)$. Consequently, we deduce that $V_{H_1}\neq V_{H_2}$ and $C_{H_1}\neq C_{H_2}$. Without loss of generality, we consider the corresponding $C$ value of the point $\Jp_2^{-1}(V_{H_2},C_{H_2})$ 
 to be less than the corresponding $C$ value of $\Jp_2^{-1}(V_{H_1},C_{H_1})$. 
In other words, if we trace the path of $\Jp(V_{\text{sol}}(C),C)$ starting from $\tilde{P}_1$, then $P_{H_1}$ is the first intersection point between $\Jp(V_{\text{sol}}(C),C)$ and $V_\infty$. Now we have two possible cases.

\underline{Case I: $C_{H_1}> C_{H_2}$.}  According to Lemma \ref{1-1 jump point} $(i)$, there exists $\kappa_1\in (1, \sqrt{\frac{2\ga}{\ga-1}})$ 
such that $C=-\kappa_1(1+V)$ intersects $\Jp(V_{\text{sol}}(C),C)$ at $P_{H_1}$.
Given the fact that $V_\infty'(C)<0$ and $V_\infty$ connects $\mathring{P}$ and $P_\infty$, $P_{H_2}$ lies below $C=-\kappa_1(1+V)$ on the $C$-coordinates. This implies $P_{H_2}$ lies on a line $C=-\kappa_2(1+V)$ with $\kappa_1<\kappa_2<\sqrt{\frac{2\ga}{\ga-1}}$. Since $\tilde{P}_1$ lies on $C = -\sqrt{\frac{2\gamma}{\gamma-1}}(1+V)$, and  $\Jp(V_\text{sol}(C),C)$ is continuous, there must exist another intersection point, which is distinct from $P_{H_2}$, between $\Jp(V_{\text{sol}}(C),C)$ and $C=-\kappa_2(1+V)$. This contradicts Lemma \ref{1 intersection}.

\underline{Case II: $C_{H_1}< C_{H_2}$.} Given that $V_\infty'(C)<0$ for $C\in (-\infty,\Cg)$, $V_{H_1}>V_{H_2}$. Therefore, there exists a $C_0\in(C_s,0)$ such that  $\frac{d \Jp_1(V_{\text{sol}},C)}{d C}|_{C=C_0}>0$. We will show that $\frac{d \Jp_1(V_{\text{sol}},C)}{d C}<0$ for any $C\in(C_s,0)$ to derive a contradiction.
By using \eqref{monotone to lines}, we obtain
\begin{align*}
	\frac{d}{d C}\Big(\frac{1+\Jp_1(V_{\text{sol}},C)}{ (1+V_{\text{sol}}(C))}\Big) = \frac{d}{d C} \Big(\frac{\ga-1}{\ga+1}+\frac{2}{\ga+1}\frac{C^2}{(1+V_{\text{sol}}(C))^2}\Big) = \frac{2}{\ga+1}\frac{C}{1+V_{\text{sol}}(C)}\frac{d}{d C}\Big(\frac{C}{1+V_{\text{sol}}(C)}\Big)<0.
\end{align*}
On the other hand, we have
\begin{align*}
	\frac{d}{d C}\Big(\frac{1+\Jp_1(V_{\text{sol}},C)}{ (1+V_{\text{sol}}(C))}\Big) =\frac{1}{1+V_{\text{sol}}(C)} \frac{d \Jp_1(V_{\text{sol}},C)}{d C} - \frac{1+\Jp_1(V_{\text{sol}},C)}{(1+V_{\text{sol}}(C))^2} \frac{d V_{\text{sol}}(C)}{d C}. 
\end{align*}
From Lemma \ref{L:FG=0}, $\frac{d V_{\text{sol}}(C)}{d C}\leq 0$ in the interval $C\in[C_g,0)$. Hence, the above identity implies $\frac{d \Jp_1(V_{\text{sol}},C)}{d C}<0$ for $C\in[C_g,0)$. For $C\in(C_s,C_g)$, we compute
\begin{align}\label{Jp_1'}
	\frac{d \Jp_1(V_{\text{sol}},C)}{d C} = \Big(\-\frac{\ga-1}{\ga+1}-\frac{2}{\ga+1}\frac{C^2}{(1+V_{\text{sol}}(C))^2}\Big) \frac{d V_{\text{sol}}(C)}{d C} + \frac{4C}{(\ga+1)(1+V_{\text{sol}}(C))}.
\end{align}
It is clear that $\frac{d \Jp_1(V_{\text{sol}},C)}{d C} <0$ when $\frac{C^2}{(1+V_{\text{sol}}(C))^2} \geq \frac{\ga-1}{2}$, so
 we reduce to studying the part of $V_{\text{sol}}(C)$ for which $(V_{\text{sol}}(C),C)$ lies in the region
\begin{align}
	\mathcal H = \mathcal H_1 \cup \mathcal H_2,
\end{align}
where
\begin{align}
	\mathcal{H}_1&=	\{(V,C) :V\geq 0 \text{ and }\frac{g_2(V)}{g_1(V)}<C^2 < \frac{\ga-1}{2}(1+V)^2\},\\
	\mathcal{H}_2&= \{ (V,C): V_0<V<0 \text{ and } \frac{V}{k(\ga,z)}<C^2<\frac{\ga-1}{2}(1+V)^2\}.
\end{align}
Here $V>k(\ga,z)C^2$ is given by Lemma \ref{lower barrier when m=1,2},  $g_1(V)$ and $g_2(V)$ are defined in \eqref{G(V,C)}, and $V_0$ is the $V$ value such that $\frac{V_0}{k(\ga,z)}=\frac{\ga-1}{2}(1+V_0)^2$.
We shall show that $\frac{d \Jp_1(V_{\text{sol}},C)}{d C}<0$ for $(V_{\text{sol}}(C),C)\in\mathcal H$. From \eqref{Jp_1'} and using $\frac{d V_{\text{sol}}(C)}{d C} >0$ for $(V_{\text{sol}}(C),C)\in\mathcal H$,  we have 
\begin{align*}
	\frac{d \Jp_1(V_{\text{sol}},C)}{d C}&< \frac{\ga-1}{\ga+1}\frac{d V_{\text{sol}}(C)}{d C} +  \frac{4C}{(\ga+1)(1+V_{\text{sol}}(C))}   \\ 	&=\frac{(\ga-1)(1+V_{\text{sol}}(C))^2G(V_{\text{sol}}(C),C)+4C(1+V_{\text{sol}}(C))F(V_{\text{sol}}(C),C)}{(\ga+1)(1+V_{\text{sol}}(C))^2F(V_{\text{sol}}(C),C)}.
\end{align*}
We denote
\begin{align*}
	\mathcal Q(V,C) :=(\ga-1)(1+V)^2G(V,C)+4C(1+V)F(V,C).
\end{align*}
We claim $\mathcal Q(V,C) <0$ for any $(V,C)\in\mathcal H$. Direct computation shows
\begin{align*}
	\mathcal Q(V,C) = &-(1+V)\Big[(m(\ga-1)+5-\ga)V^2+[m(\ga-1-2(\ga^2-2\ga-1)z)+9-\ga]V+2m(\ga+1)z+4\Big]C^2\\
	&+4(1+mz +V)C^4-(\ga-1)(1+V)^3(1+V+m\ga z)V.
\end{align*}
As  $\mathcal Q(V,C)$ is a quadratic function of $C^2$ with positive leading order coefficient, it is sufficient to verify the sign at the end-points of the $C^2$ intervals defining $\mathcal H_1$ and $\mathcal H_2$. For convenience, due to the complexity of $k(\ga,z) = \frac{V_8(\ga,z)}{C_8^2(\ga,z)}$, we first enlarge $\mathcal H_2$ by observing
\begin{align*}
	\frac{V}{k(\ga,z)} \geq \frac{V}{k(\ga,z_M)} = -\frac{\ga}{2+\sqrt{2\ga}}V > -\frac{\ga V}{4}.
\end{align*}
To show $\mathcal Q(V,C)<0$ in $\mathcal H$, it is therefore sufficient to prove $\mathcal Q(V,-\sqrt{\frac{\ga-1}{2}}(1+V))<0$ for $V\geq  V_0$, $\mathcal Q(V,\sqrt{\frac{V(1+V)(1+m\ga z+V)}{(m+1)V+2mz}})\leq 0$ for $V\geq 0$, and $\mathcal Q(V,-\sqrt{-\frac{\ga V}{4}})<0$ for $V\in(V_0,0)$. As $V_0$ is not a convenient lower bound to work with, we observe that the intersection of the curves $\{C = -\sqrt{\frac{\ga-1}{2}}(1+V)\}$ and $\{C= -\sqrt{-\frac{\ga V}{4} }\}$ with $V\in( -1,0)$ occurs at  $\hat V<V_0$ with $\hat V=\frac{4-5\ga+\sqrt{(9\ga-8)\ga}}{4(\ga-1)}= \frac{4(\ga-1)}{4-5\ga-\sqrt{(9\ga-8)\ga}}>-\frac{2(\ga-1)}{3\ga-2}$.

To show 
$\mathcal Q(V,-\sqrt{-\frac{\ga V}{4}})<0$ with $V\in(V_0,0)$, we write $\mathcal Q(V,-\sqrt{-\frac{\ga V}{4}})=\frac{V}{4} p(V) $ where $p(V)$ is a fourth order polynomial as defined in Lemma \ref{mathcal Q<0}.  The positivity of $p(V)$ in the interval $V\in( \hat V,0)$ is shown in Lemma \ref{mathcal Q<0}.

Next, as  $G(V,\sqrt{\frac{V(1+V)(1+m\ga z+V)}{(m+1)V+2mz}})=0$ and $F(V,C)>0$ on $\mathcal H_1$, we see $\mathcal Q(V,\sqrt{\frac{V(1+V)(1+m\ga z+V)}{(m+1)V+2mz}})\leq 0$ for $V\geq 0$. 

For $\mathcal Q(V,-\sqrt{\frac{\ga-1}{2}}(1+V))$, we compute
\begin{align*}
	\mathcal Q(V,-\sqrt{\frac{\ga-1}{2}}(1+V)) = \frac{\ga-1}{2}(1+V)^3&\Big[[3\ga-9-m(\ga-1)]V^2+\Big[5\ga-15-m[\ga-1+(-2\ga^2+4\ga+4)z]\Big]V\\
	& \ +2\ga-6-4mz\Big] =:\frac{\ga-1}{2}(1+V)^3 q(V).
\end{align*}
We note that $q(V)$ is a quadratic function in $V$ with a negative leading coefficient. We have $q(-1) = 2m\ga(2-\ga)z>0$, and 
\begin{align*}
	q(-\frac{2(\ga-1)}{3\ga-2}) =  \frac{2\ga}{(2-3\ga)^2} \Big[\ga-3+m(\ga-1)^2-2m(3\ga-2)(\ga^2-3\ga+3)z\Big]<0, 
\end{align*}
where we used $\ga-3+m(\ga-1)^2\le \ga-3+2(\ga-1)^2 <0$ for any $\ga\in(1,\gaSix]$ and $\ga^2-3\ga+3>0$ in the last inequality. It follows that $q(V)<0$ for any $V>-\frac{2(\ga-1)}{3\ga-2}$, and hence for $V\geq V_0$. Hence, $\mathcal Q(V,-\sqrt{\frac{\ga-1}{2}}(1+V))<0$  for any $(V,C)\in\mathcal H$.

This completes the proof.
\end{proof}

\section{Proof of the Main Theorem \ref{main thm}}

$(i):$ The real analyticity locally around the origin (i.e. $x=0$) follows from Theorem \ref{analytic} and Lemma \ref{L: finite slope in x}. To see the  existence of the maximal extension of the solution onto $(0,x_{s}]$, we argue as follows. From Proposition~\ref{approach C=-1-V}, the solution trajectories in the $(V,C)$ plane extend from the origin until they intersect the lower sonic line at a point $(-C_s-1,C_s)$ with finite $C_s$. The maximal value $x_s$ is then determined by the relation $C(x_s)=C_s$. To see that $x_s$ is finite, we introduce the variable transformation $y = \ln x$. The monotonicity of $C(x)$ allows us to compute the derivative of $y$ with respect to $C$. This is given by $\frac{dy}{dC} =-\la \frac{D(V(C),C)}{F(V(C),C)}$. Since $F(V(C),C)<0$ and $D(V,C)$ is bounded in $\mathcal S_U\cap \{V\geq V_8\}$, integrating $\frac{dy}{dC}$ backwards from $C_0\in(C_s,0)$ to $C_s$ determines $y_{s}$ to be finite, and hence $x_{s}$ is also finite. To see that $x_s$ gives the maximal \textit{smooth} extension of the flow, observe that, as a function of $x$, trajectories can only pass smoothly through the sonic line with $C<0$ at the triple points $P_7$ and $P_8$. As we have established $C_s<\min\{C_7,C_9\}$, the flow cannot be extended smoothly past $x_s$.
The existence of $C_H\in(-\infty, \Cg)$ is given by Lemma \ref{PH} and the jump location $x_H\in(0,x_s)$ is determined via $\Jp(V(x_H),C(x_H))=(V_H,C_H)$.

\noindent $(ii):$ The existence and finiteness of $x_s$ follows as in part \textit{(i)}. The uniqueness of the intersection point $P_H$ is shown in Lemma \ref{uniq PH}.

\

{\bf Acknowledgements.} JJ and JL are supported in part by the NSF grants DMS-2009458 and DMS-2306910. MS is supported by the EPSRC Post-doctoral Research Fellowship EP/W001888/1.

\begin{appendices}
\section{Proof of Lemma \ref{L:FG=0}} \label{Ap: FG=0}

\textit{{{Proof of (i):}}} We observe first that, for any $(V,C)\in\mathcal S_U$, $C\in(-(1+V),0)$ and $V>-1$. Hence, recalling~\eqref{F(V,C)}, as $f_1(V)>0$,
\begin{align}\label{Ieq: f2(V8)>0}
	\frac{F(V,C)}{C} <\frac{F(V,C)}{C}\Big|_{C=-(1+V)}= \frac{m(\ga-1)}{2}\Big[ -(1+V)^2+[1+(\ga-2)z](1+V)-\ga z\Big].
\end{align}
Given that $F(V_8,1+V_8)=F(V_8,C_8)=0$, and $V_6$ and $V_8$ are two roots of the polynomial $-(1+V)^2+[1+(\ga-2)z](1+V)-\ga z$ with $V_6\leq V_8$, $\frac{F(V,C)}{C}\Big|_{C=-(1+V)} <0$ for any $V\in \mathcal S_U$ with $V\geq V_8$. Hence, $F(V,C)>0$ in $\mathcal S_U\cap\{V\geq V_8\}$.

In the region $\mathcal S_L\cap \{V\geq V_8\}$, we consider the root branch in $C(V)$ form, explicitly given by $C_F(V) = -\sqrt\frac{f_2(V)}{f_1(V)}$. Since $V\in(V_8,0)$, $f_1(V)> 0$ and $f_2(V)>0$ by checking that 
\begin{equation}\label{f2>0}
\begin{aligned}
	&(1+V_8)^2f_1(V_8)-f_2(V_8)=0,\text{ so that } f_2(V_8)>0,\\	
	&f_2'(V) = 2a_1(1+V)-a_2\geq 2a_1[1+V_8(z_M)]-a_2 >\frac{4\sqrt \ga + m(\ga-1)(\sqrt \ga -\sqrt 2)}{2(\sqrt 2+\sqrt \ga)}>0
\end{aligned}
\end{equation}
for any $m=1,2$, $V\in[V_8,0)$, $z\in(0,z_M]$ and $\ga\in(1,3]$. Hence, $f_2(V)>0$ for any $V\in[V_8,0)$. This implies that $C_F(V) = -\sqrt\frac{f_2(V)}{f_1(V)}$ is well-defined in the interval $V\in[V_8,0)$. Now, we shall show that $C_F'(V)<0$. Since $C_F(V)< 0$ for $V\in[V_8,0)$, by using \eqref{f2>0}, we have
\begin{align*}
	C_F'(V) = \frac{1}{2C_F(V)f_1^2(V)}[f_2'(V)f_1(V)+\frac{mz f_2(V)}{(1+V)^2}]<0.
\end{align*}
Hence, by the Inverse Function Theorem, we  conclude that $V_F^+(C)$ satisfies the claimed properties.

\textit{{{Proof of (ii):}}} 
Recalling \eqref{G(V,C)}, $G(V,C)=0$ is given by
\begin{align}\label{G=0}
	&C^2[(m+1)V+2mz]-V(V+1)(V+1+m\ga z)=0.
\end{align}
As $C\to -\infty$, it is straightforward to see that there are exactly three root branches satisfying $\lim_{C\to -\infty}V_G^-(C) = -\infty$, $\lim_{C\to -\infty}V_G(C) = -\frac{2mz}{m+1}$, and $\lim_{C\to -\infty}V_G^+(C) = \infty$. At $C=0$, we have $V_G^-(0) = -1-m\ga z$, $V_G(0)=-1$ and $V_G^+(0)=0$. We are only interested in $V_G(C)$ and $V_G^+(C)$. Moreover, it is clear that there is no corresponding $C$ value for $V\in(-1-m\ga z,-1)\cup[-\frac{2mz}{m+1},0) $. Therefore, the equation $G(V,C)=0$ has exactly three distinct roots branches $V_G^-(C)<V_G(C)<V_G^+(C)$. Now, we check that $V_G'(C)<0$ and $(V_G^+)'(C)<0$ for $C< 0$. By taking the derivative $\frac{d}{dC}$ on both sides of \eqref{G=0}, we obtain
\begin{align*}
	2C[(m+1)V_G(C)+2mz]+((m+1)C^2-3V_G(C)^2-2(2+m\ga z)V_G(C)-1-m\ga z)V_G'(C)=0,
	\end{align*}
	and hence
\begin{align*} V_G'(C) = \frac{-2C[(m+1)V_G(C)+2mz]}{(m+1)C^2-3V_G(C)^2-2(2+m\ga z)V_G(C)-1-m\ga z}=\frac{-2V_G(C)C[(m+1)V_G(C)+2mz]}{-2mz C^2-2V_G^3(C)-(2+m\ga z)V_G^2(C)}<0
\end{align*}
where we used $V_G(C)\in(-1,-\frac{2mz}{m+1})$ in the last equality and last inequality. Hence, $V_G'(C)<0$ for $C< 0$. A similar argument also shows that $(V_G^+)'(C)<0$ for $C< 0$. It only remains to show that $(V_G^+(C),C)\in \mathcal S_U$. This can be verified by examining the asymptotic behavior of $V_G^+(C)$ as $C$ approaches infinity. This behavior  resembles that of a linear function, specifically, $V(C)=-C$.

\section{Quantitative Properties}

\begin{lemma}\label{lower bound of zg}
Let $m=1,2$. For any $\ga\in[\gau,3]$, $z_g(\ga)\geq\frac{1}{10}$.
\end{lemma}
\begin{proof}
	When $m=1$, we check
	\begin{align*}
	(z_g(\ga)-\frac{1}{10})\ga(\ga-1) = \sqrt{\gamma^2+(\gamma-1)^2}-\gamma-\frac{\ga(\ga-1)}{10}
	=\frac{-(\ga-1)(\ga^3+19\ga^2-100\ga+100)}{100( \sqrt{\gamma^2+(\gamma-1)^2}+\gamma+\frac{\ga(\ga-1)}{10})}>0,
	
\end{align*}
where we used $\ga^3+19\ga^2-100\ga+100<0$ for $\ga\in[\gau,3]$, which is shown in Proposition~\ref{prop:poly} \eqref{A-1}, to conclude the inequality. 

	When $m=2$, we compute
	\begin{align*}
		(z_g(\ga)-\frac{1}{10})\gamma[4\gamma(\gamma-1)+\frac{8}{3}]
		&= \frac{-6\ga^3-24\ga^2+11\ga-15}{15} + \sqrt{(2\gamma^2-\gamma+1)^2 +2\gamma(\gamma-1)[4\gamma(\gamma-1)+\frac{8}{3}]}.
	\end{align*}
We claim that this expression is positive. Obviously, $-6\ga^3-24\ga^2+11\ga-15<0$. Therefore, it is sufficient to check the negativity of
	\begin{align*}
		q(\ga) &= ( \frac{-6\ga^3-24\ga^2+11\ga-15}{15}  )^2-(2\gamma^2-\gamma+1)^2 -2\gamma(\gamma-1)[4\gamma(\gamma-1)-\frac{8}{3}\\
		& = \frac{4\ga}{225}(\ga-3)[3\ga(\ga-1)+2](3\ga^2+36\ga-55).
	\end{align*}
Since $\gau$ is a root of $3\ga^2+36\ga-55$, and $3\ga^2+36\ga-55\geq 0$ for each $\ga\in[\gau,3]$, we conclude the desired positivity.\end{proof}

\begin{lemma}\label{z0=zg}
Let $m=1,2$. There exists a unique $\ga\in(1,2]$ such that $z_0(\ga)=z_g(\ga)$. Moreover, $z_g<z_0$ when $\ga\in(1,\gau)$ and $z_0\leq z_g $ when $\ga\in[\gau,2]$. Moreover, $\gau\in(\frac{79}{50},\frac{159}{100})$ when $m=1$, and  $\gau\in(\frac{77}{50},\frac{31}{20})$ when $m=2$.
\end{lemma}

\begin{proof}
When $m=1$, we check
\begin{align*}
	z_0(\ga)-z_g(\ga) &=\frac{22-5\ga}{125} - \frac{\sqrt{\ga^2+(\ga-1)^2}-\ga}{\ga(\ga-1)} = \frac{\ga(-5\ga^2+27\ga+103)-125\sqrt{\ga^2+(\ga-1)^2}}{125\ga(\ga-1)}\\
	&=\frac{25\ga^5-245\ga^4-546\ga^3+5016\ga^2-15625\ga+15625}{125\ga[\ga(-5\ga^2+27\ga+103)+125\sqrt{\ga^2+(\ga-1)^2}]}.
\end{align*}
By Proposition~\ref{prop:poly} \eqref{L A.2 -1}, the polynomial numerator only has one root in the interval $\ga\in(1,2]$. We denote it by $\gau$ and $\gau\in(\frac{79}{50},\frac{159}{100})$ by Theorem  \ref{Sturm's theorem}.

When $m=2$, we compute
\begin{align*}
&z_0(\ga)-z_g(\ga)\\ &=\frac{22-5\ga}{125} -\cfrac{\sqrt{(2\gamma^2-\gamma+1)^2 +2\gamma(\gamma-1)[4\gamma(\gamma-1)+\frac{8}{3}]}-(2\gamma^2-\gamma+1)}{\gamma[4\gamma(\gamma-1)+\frac{8}{3}]}\\
&=\frac{-60\ga^4+324\ga^3+446\ga^2-199\ga+375-125\sqrt{3}\sqrt{(3\ga-1)^2(4\ga^2-4\ga+3)}}{375\gamma[4\gamma(\gamma-1)+\frac{8}{3}]}\\
&=\frac{8(450\ga^7 -4860\ga^6 +6432\ga^5 +39111\ga^4-207817\ga^3+359749\ga^2-275503\ga+110250)}{375[4\gamma(\gamma-1)+\frac{8}{3}][-60\ga^4+324\ga^3+446\ga^2-199\ga+375+125\sqrt{3}\sqrt{(3\ga-1)^2(4\ga^2-4\ga+3)}]}.
\end{align*}
By Proposition~\ref{prop:poly} \eqref{L A.2 -2}, the polynomial only has one root in the interval $\ga\in(1,2]$. We denote it by $\gau$ and $\gau\in(\frac{77}{50},\frac{31}{20})$ by Theorem  \ref{Sturm's theorem}.
\end{proof}

\begin{lemma}\label{p(z)<0}
For any $\ga\in(1,2]$ and $z\in[z_s,z_M]$, 
\begin{align*}
	p(z)=(\ga-2)^3(\ga-1)z^3+2(2-\ga)(\ga^2+\ga+1)z^2+(\ga^2-4\ga-6)z+1<0.
\end{align*}
\end{lemma}
\begin{proof}
	By using $z<z_M<\frac{1}{5}$, we have
	\begin{align*}
		p'(z) &= 3(\ga-2)^3(\ga-1)z^2+4(2-\ga)(\ga^2+\ga+1)z+\ga^2-4\ga-6 \\
		&< \frac{4}{5}(2-\ga)(\ga^2+\ga+1)+\ga^2-4\ga-6 =-\frac{4\ga^3-9\ga+16\ga+22}{5}<0
	\end{align*}
	for any $\ga\in(1,2]$. Hence, for any $\ga\in(1,\gau)$,
	\begin{align*}
		p(z)\leq p(z_0) = \frac{-125\ga^7+2525\ga^6-27310\ga^5+154918\ga^4-487091\ga^3+1084814\ga^2-1166290\ga+217809}{1953125}
	\end{align*}
	which is negative by  Proposition~\ref{prop:poly} \eqref{L A.3-1}.
	
	When $\ga\in[\gau,2]$, by using Lemma \ref{z0=zg} and $z_0(\ga)$ is decreasing in $\ga$, we have $z_g(\ga)\leq z_0(\ga)<z_0(\frac{8}{5})$, therefore:
	\begin{align*}
	p(z)\leq p(z_g)\leq p(z_0)<p(z_0(\frac{8}{5})) = \frac{2744\ga^4-68208\ga^3+317142\ga^2-880880\ga+760577}{1953125}<0,
	\end{align*}
	where the negativity of the numerator is proven in Proposition~\ref{prop:poly} \eqref{L A.3-2}.
\end{proof}
\begin{lemma}\label{w(zs)<1/2}
For any $\ga\in(1,2]$ and $z\in[z_s(\ga),z_M(\ga)]$, $w(z)<\frac{1}{2}$.
\end{lemma}
\begin{proof}
	Since $w'(z)<0$ and $z_0(\ga)\leq z_s(\ga)$ for $\ga\in(1,2]$, it is enough to check $w^2(z_0)<\frac{1}{4}$.
	Direct computation shows, for any $\ga\in(1,2]$,
	\begin{align*}
		w^2(z_0(\ga))-\frac{1}{4} = \frac{100\ga^4-1280\ga^3+10856\ga^2-23264\ga+10619}{62500}<0,
	\end{align*}  
	where the negativity of the numerator is shown in Proposition~\ref{prop:poly} \eqref{L A.4}.
\end{proof}

\begin{lemma}\label{w(zs)>2/5}
For any $\ga\in(1,\gau)$, $\frac{2}{5}<w(z_s(\ga))$.
\end{lemma}
\begin{proof}
Direct computation shows
\begin{align*}
	w(z_s(\ga))^2-\frac{4}{25} = \frac{25\ga^4-320\ga^3+2714\ga^2-5816\ga+4061}{15625}>0,
\end{align*}
where the positivity of the numerator is proven in Proposition \ref{prop:poly} \eqref{L A.5}.
\end{proof}

\begin{lemma}\label{p(zs)>0}
Let $m=1,2$. For any $\ga\in(1,\gau)$, \eqref{p(z)>p(zs)>0} holds.
\end{lemma}
\begin{proof}

When $m=1$, we have, for any $\ga\in(1,\gau)$,
\begin{align*}
p(z_0(1)) &=  \frac{1836\ga^2-56125\ga-11594+(91+108\ga)\sqrt{289\ga^2+15844\ga+109156}}{27000}\\
&=\frac{4(34\ga-57)(108\ga^2-233\ga-125)}{(91+108\ga)\sqrt{289\ga^2+15844\ga+109156}-1836\ga^2+56125\ga+11594}>0.

\end{align*}
When $m=2$, for each $\ga\in(1,\gau)$,
\begin{align*}
p(z_0(1)) &=  \frac{5219\ga^2-121500\ga+4749+(239+307\ga)\sqrt{289\ga^2+11594\ga+39531}}{76750}\\
&= \frac{(34\ga-57)(921\ga^2-2046\ga-511)}{(239+307\ga)\sqrt{289\ga^2+11594\ga+39531}-5219\ga^2+121500\ga-4749}>0.

\end{align*}
\end{proof}

\section{Applications of Fourier-Budan Theorem} 
\begin{theorem}\cite[Fourier-Budan Theorem]{Bu07,Fourier20}\label{Fourier-Budan Theorem}
Let $f(x)=0$ be a polynomial equation of degree $n$ with real coefficients. After ignoring all zeros, let $V_c$ denote the number of variations of sign in the sequence $\{f(x),f'(x),f''(x),\dots, f^{(n)}(x)\}$ when $x=c$, where $c$ is a real number. Let $a$ and $b$ be two real numbers and $a<b$. Then the number of roots of $f(x)=0$ in the interval $x\in[a,b]$ is $|V_a-V_b|-2k$, where $k$ is a positive integer or zero.
\end{theorem}

In any sequence of signs, we shall use $*$ to denote an indeterminate sign or zero.

\begin{lemma}\label{one zero-1}
Let $m=1,2$. For any $\ga\in(1,\gau)$ and $z\in(z_g,\frac{1}{10})$, the polynomial $\Barrier(C)=0$ defined in \eqref{6th poly} only has one zero in the interval $C\in[-1,0]$.
\end{lemma}
\begin{proof}
A direct computation and simple estimates shows that the sequence of signs for $(\Barrier(0),\ldots,\Barrier^{(6)}(0))$ is $(-,+,*,-,*,+,+)$, where $*$ denotes an indeterminate sign. The number of sign changes, therefore, is $V_0=3$. Similarly, at $C=-1$, we find the sequence of signs is $(+,-,-,+,*,-,+)$, and  thus $V_{-1}=4$. By Theorem \ref{Fourier-Budan Theorem}, we conclude that $\Barrier(C)$ has only one zero in the interval $C\in[-1,0]$.
\end{proof}

\begin{lemma}\label{mathcal Q<0}
Let $m=1,2$. For any $\ga\in(1,2]$ and $z\in(z_g,z_M]$,
\begin{align*}
p(V)=&-4(\ga-1)V^4+[(m-1)\ga^2-(m+11)\ga+16+4m\ga(1-\ga)z]V^3+[(2m-1)\ga^2-(2m+10)\ga+24\\
	&-2m\ga(\ga^2+4\ga-7)z]V^2+[m\ga^2-(3+m)\ga+16-m\ga(2\ga^2+5\ga-16)z]V+4+2m\ga(3-\ga)z
\end{align*}
is positive in the interval $V\in( \hat V,0)$ where $\hat V=\frac{4-5\ga+\sqrt{(9\ga-8)\ga}}{4(\ga-1)}$ is defined in Lemma \ref{uniq PH}.
\end{lemma}

\begin{proof}
It is obvious that $\hat V>-1$. We will first show that $p(V)$ only has one zero in the interval $V\in(-1,0)$. We check that at $V=-1$, the the sequence of signs for $(p(-1),\ldots,p^{(4)}(-1))$ is $(-,-,+,+,-)$. Hence, $V_{-1}=2$. And, the sequence of signs at $V=0$ is given by $(+,+,*,*,-)$.
Hence, either $V_0=1$ or $3$. 
For both cases, Theorem \ref{Fourier-Budan Theorem} tells that $p(V)$ has only one zero in the interval $V\in(-1,0)$. Since $p(-1)<0$ and $p(0)>0$, it suffices to check the positivity of $p(\hat V)$. Since  $C = -\sqrt{\frac{\ga-1}{2}}(1+V)$ intersects $C= -\sqrt{-\frac{\ga V}{4} }$ at $V = \hat V$, $\mathcal Q(\hat V,-\sqrt{\frac{\ga-1}{2}}(1+\hat V)) = \mathcal Q(\hat V,-\sqrt{-\frac{\ga \hat V}{4} })$. Hence, $p(\hat V)>0$.
\end{proof}

\section{Polynomials}

In this appendix, we collect the polynomial sign conditions used throughout the paper for regular polynomials. The proofs of all of these inequalities follow from Sturm's Theorem, which we state below for the convenience of the reader.

\begin{theorem}[\cite{Sturm09}, Sturm's Theorem]\label{Sturm's theorem}
	Take any polynomial $p(x)$, and let $p_0(x), \ldots p_m(x)$ denote the Sturm chain corresponding to $p(x)$. Take any interval $(a, b)$ such that $p_i(a), p_i(b) \neq$ 0 , for any $i\in\{0,1,\dots,m\}$. For any constant $c$, let $\sigma(c)$ denote the number of changes in sign in the sequence $p_0(c), \ldots p_m(c)$. Then $p(x)$ has $\sigma(a)-\sigma(b)$ distinct roots in the interval $(a, b)$.
\end{theorem}

\begin{proposition}\label{prop:poly}
\begin{enumerate}[(i)]
\item\label{Poly L 3.4} Let $m=1,2$. For any $\ga\in[\ga_g,3]$, $m[(3-\frac{209\ga}{729}-\frac{152\ga^2}{729})z_M+\frac{2888\ga-9044}{19683}]-\frac{703}{19683}<0.$

\item\label{r(z0)<0} For any $\ga\in(1,\gau)$, $1953125r(z_0)= -250\ga^7+5050\ga^6-51495\ga^5+266711\ga^4-673057\ga^3+1369003\ga^2-1769830\ga+771743<0$.

\item\label{r(z0(1.59))<0} Let $m=1$. For any $\ga\in[\ga_u,2)$, $7812500000r(z_0(\frac{159}{100})) =22188041\ga^4-451420037\ga^3+1178808488\ga^2-7162470820\ga+9959809328<0$.

\item\label{r(z0(1.55))<0} Let $m=2$. For any $\ga\in[\ga_u,2)$, $62500000r(z_0(\frac{31}{20}))=3[61731\ga^4-1244367\ga^3+3215408\ga^2-19360620\ga+26076848]<0$.
\item\label{L 3.7-1}
For each $\ga\in(1,2]$, $15625 q(z_0(2))=2592\ga^4-37368\ga^3+341118\ga^2-1143750\ga+828125<0$

\item\label{L 3.7-2}
For each $\ga\in(1,2]$, $q(1)=18\ga^4-90\ga^3+123\ga^2-141\ga+53<0.$

\item\label{L 3.13-1}
For each $\ga\in(1,2]$, $-24(\ga-2)^2z_0^2+(92+54\ga)z_0-21=\frac{-600\ga^4+7680\ga^3-68886\ga^2+158584\ga-121589}{15625}<0.$

\item\label{L 3.13-2}
For each $\ga\in(1,2]$, $15625p(z_0)=450\ga^4-5760\ga^3+48852\ga^2-89063\ga+49348>0.$

\item\label{L 3.13 p3}
For each $\ga\in(1,\gau)$, $g(\frac{3}{20})=\frac{81\ga^4-1119\ga^3+3476\ga^2-3248\ga+1048}{400}>0.$

\item\label{A-1}
For each $\ga\in[\ga_g,3]$, $\ga^3+19\ga^2-100\ga+100<0.$

\item\label{L 3.17 p1}
For each $\ga\in(1,2]$, $36\ga^4-180\ga^3+167\ga^2+405\ga-414>0.$

\item\label{L 3.17 p2}
For each $\ga\in(1,2]$, $A(\ga,\frac{1}{5}) =\frac{ 64\ga^7-352\ga^6-292\ga^5+4270\ga^4-3575\ga^3-11948\ga^2+15074\ga+4729}{3125}>0.$

\item\label{q(z0)<0}
For any $\ga\in(1,\gau)$, $(20-11\ga)z_0+5-2\sqrt 3-7w(z_0)<0$.

\item\label{p(2(ga-1)/11)<0}
For any $\ga\in(1,\gau)$, $p(\frac{2(\ga-1)}{11})<0$ for $p(z)$ defined in \eqref{p(z)}. 
\item\label{L A.2 -1}
For any $\ga\in(1,2]$, $ 25\ga^5-245\ga^4-546\ga^3+5016\ga^2-15625\ga+15625 $ has only one real root, $\gau\in(\frac{79}{50},\frac{159}{100})$. 
\item\label{L A.2 -2}
For any $\ga\in(1,2]$, $450\ga^7 -4860\ga^6 +6432\ga^5 +39111\ga^4-207817\ga^3+359749\ga^2-275503\ga+110250$ has only one real root, $\gau\in(\frac{77}{50},\frac{31}{20})$. 
\item\label{L A.3-1}
For any $\ga\in(1,\gau]$, $-125\ga^7+2525\ga^6-27310\ga^5+154918\ga^4-487091\ga^3+1084814\ga^2-1166290\ga+217809<0$.

\item\label{L A.3-2}
For any $\ga\in[\gau,2]$, $2744\ga^4-68208\ga^3+317142\ga^2-880880\ga+760577<0$.

\item\label{L A.4}
For any $\ga\in(1,2]$, $100\ga^4-1280\ga^3+10856\ga^2-23264\ga+10619<0$.

\item\label{L A.5}
For any $\ga\in(1,\gau)$, $25\ga^4-320\ga^3+2714\ga^2-5816\ga+4061>0$.
\end{enumerate}
\end{proposition}

\begin{proof}
Each of these statements is proved directly from Theorem~\ref{Sturm's theorem} after tedious calculation, with the exception of~\eqref{q(z0)<0}. For this inequality, direct computation shows that
	\begin{align*}
		&(20-11\ga)z_0+5-2\sqrt 3-7w(z_0)\\ &= \frac{55\ga^2-342\ga-250\sqrt{3}+1065 - 
			7 \sqrt{    25 \ga^4- 320 \ga^3+ 2714 \ga^2- 5816 \ga+6561 }}{125}\\
			&= -\frac{4}{125}\frac{-450\ga^4+ 5485\ga^3+ (-25282 + 6875\sqrt{3}) \ga^2+ (110869 - 42750 \sqrt{3}) \ga+3 (-83353 + 44375 \sqrt{3})}{55\ga^2-342\ga-250\sqrt{3}+1065 + 7 \sqrt{ 25 \ga^4- 320 \ga^3+ 2714 \ga^2- 5816 \ga+6561}}\\
			&<0,
	\end{align*}
	where we used, for any $\ga\in(1,2]$, $55\ga^2-342\ga-250\sqrt{3}+1065>0$ and the positivity of the numerator, which can be shown by Theorem \ref{Sturm's theorem}, to conclude the last inequality. 
\end{proof}

\end{appendices}

\end{document}